\documentclass[oneside]{amsart}
\usepackage{amssymb, amsmath, amsthm}
\usepackage{graphicx}
\usepackage[all]{xy}
\xyoption{arc}
\CompileMatrices

\DeclareMathAlphabet{\mathbbold}{U}{bbold}{m}{n}

\def\k{\mathbbold{k}}

\DeclareSymbolFont{rsfscript}{OMS}{rsfs}{m}{n}
\DeclareSymbolFontAlphabet{\mathrsfs}{rsfscript}

\DeclareFontFamily{OMS}{rsfs}{\skewchar\font'177}
\DeclareFontShape{OMS}{rsfs}{m}{n}{%
      <5> rsfs5
      <6> <7> rsfs7
      <8> <9> <10> rsfs10
      <10.95> <12> <14.4> <17.28> <20.74> <24.88> rsfs10
      }{}

\def\calA{\mathrsfs{A}}
\def\calB{\mathrsfs{B}}

\def\calD{\mathrsfs{D}}
\def\calE{\mathrsfs{E}}
\def\calF{\mathrsfs{F}}

\def\calM{\mathrsfs{M}}
\def\calN{\mathrsfs{N}}
\def\calO{\mathrsfs{O}}

\def\calQ{\mathrsfs{Q}}
\def\calR{\mathrsfs{R}}
\def\calS{\mathrsfs{S}}
\def\calT{\mathrsfs{T}}
\def\calU{\mathrsfs{U}}
\def\calV{\mathrsfs{V}}
\def\calW{\mathrsfs{W}}

\newcommand{\comment}[1]{}

\newcommand{\smallcaps}[1]{\textrm{\textsc{#1}}}
\newcommand{\brac}[1]{\left(#1\right)}
\newcommand{\sbrac}[1]{\left[#1\right]}

\newcommand{\cbrac}[1]{{\left\{#1\right\}}}

\newcommand{\sym}{\mathbb{S}}
\newcommand{\symn}{\sym_n}
\newcommand{\dt}{\bullet}
\newcommand{\clc}{,\ldots,}
\newcommand{\smtree}[1]{\vcenter{\xymatrix@C=0.6pc@R=0.6pc{#1}}}

\newcommand{\bdY}{\mathbf{Y}}
\newcommand{\NAPY}{\smallcaps{NAP}_Y}
\newcommand{\RT}{\smallcaps{RT}}
\newcommand{\PFs}{\smallcaps{PF}_\ast}
\newcommand{\NAP}{\smallcaps{NAP}}
\newcommand{\Cact}{\smallcaps{Cact}}
\newcommand{\BCact}{\smallcaps{BCact}}

\newcommand{\BCY}{\BCact_Y}
\newcommand{\ed}[1]{\overrightarrow{#1}}

\newcommand{\PLC}{\text{Pl}\smallcaps{Cact}}
\newcommand{\LR}{\smallcaps{LR}}

\DeclareMathOperator{\As}{As}
\DeclareMathOperator{\Lie}{Lie}
\DeclareMathOperator{\Com}{Com}
\DeclareMathOperator{\PreLie}{PreLie}
\DeclareMathOperator{\PostLie}{PostLie}
\DeclareMathOperator{\Mag}{Mag}

\DeclareMathOperator{\ComTrias}{ComTrias}
\DeclareMathOperator{\CTD}{CTD}
\DeclareMathOperator{\Zinb}{Zinb}
\DeclareMathOperator{\Leib}{Leib}
\DeclareMathOperator{\Perm}{Perm}

\DeclareMathOperator{\Vect}{Vect}
\DeclareMathOperator{\gVect}{gVect}

\DeclareMathOperator{\SPAN}{span}

\DeclareMathOperator{\one}{\mathbbold{1}}

\DeclareMathOperator{\Hom}{Hom}

\DeclareMathOperator{\Set}{Set}

\theoremstyle{plain}
\newtheorem{theorem}{Theorem}[section]

\newtheorem{proposition}[theorem]{Proposition}
\newtheorem{corollary}[theorem]{Corollary}

\theoremstyle{definition}
\newtheorem{definition}{Definition}
\newtheorem{example}{Example}[section]

\theoremstyle{remark}
\newtheorem{remark}{Remark}[section]

\numberwithin{equation}{subsection}

\title{Cacti and filtered distributive laws}
\author{Vladimir Dotsenko}

\address{Mathematics Research Unit, University of Luxembourg, Campus Kirchberg, 6, Rue Richard Coudenhove-Kalergi, L-1359 Luxembourg, Grand Duchy of Luxembourg}

\email{vladimir.dotsenko@uni.lu}

\author{James Griffin}

\address{School of Mathematics, University of Southampton, Highfield, Southampton, SO17 1BJ, United Kingdom}

\email{J.T.Griffin@soton.ac.uk}

\begin{document}

\begin{abstract}
Motivated by the second author's construction of a classifying space for the group of pure symmetric automorphisms of a free product, we introduce and study a family of topological operads, the operads of based cacti, defined for every pointed topological space~$(Y,\bullet)$. These operads also admit linear versions, which are defined for every augmented graded cocommutative coalgebra~$C$. We show that the homology of the topological operad of based $Y$-cacti is the linear operad of based $H_*(Y)$-cacti. In addition, we show that for every coalgebra $C$ the operad of based $C$-cacti is Koszul. To prove the latter result, we use the criterion of Koszulness for operads due to the first author, utilising the notion of a filtered distributive law between two quadratic operads. We also present a new proof of that criterion which works over a ground field of arbitrary characteristic.
\end{abstract}
\maketitle

\section{Introduction}

One of the most famous algebraic operads of topological origin is the operad of Gerstenhaber algebras, which is the homology operad of the topological operad of  little $2$-disks~\cite{CLM,GJ}. The $k^\text{th}$ component of the operad of little $2$-discs is homotopy equivalent to the configuration space of $k$ ordered points in~$\mathbb{R}^2$ whose fundamental group is the pure braid group on $k$ strands. One natural way to generalise braid groups is to consider configurations of subsets that have more interesting topology than points. The simplest example of these ``higher-dimensional'' versions of braid groups is given by ``groups of loops'', the $n^\text{th}$ one being the group of motions of $n$ unknotted unlinked circles in~$\mathbb{R}^3$ bringing each circle to its original position. Alternatively, these groups can be viewed as groups of pure symmetric automorphisms of the free group with $n$ generators, that is automorphisms sending each generator to an element of its conjugacy class. The integral cohomology of these groups was computed by Jensen, McCammond and Meier in~\cite{JMM06}; that paper also contains references and historical information on this group. The description of the cohomology algebras in~\cite{JMM06} looks very similar to that for pure braid groups~\cite{Arn69}. Moreover, as a symmetric collection, the collection of cohomology algebras is isomorphic to $\Com\circ\PreLie_1$ which bears striking resemblance with the isomorphism $e_2\simeq\Com\circ\Lie_1$ for the operad of Gerstenhaber algebras. However, there is no natural operad structure on the collection of homology groups of the groups of loops. 

In \cite{Griffin}, the second author computed the cohomology of the groups of pure symmetric automorphisms in a different way, as a particular case of a much more general result: for an arbitrary $n$-tuple of groups $(G_1,\ldots,G_n)$, he computed the cohomology of the Fouxe-Rabinovitch group $\smallcaps{FR}(G)$ of partial conjugation automorphisms of the free product $G=G_1\ast\cdots\ast G_n$. For that, he used a construction of a classifying space of that group via a moduli space of ``cactus products'' of the classifying spaces $Y_i=BG_i$. In the case when $G_1=G_2=\ldots=G_n$, these spaces form a symmetric collection, but alas do not form a topological operad either. However, it turns out that they admit a slight modification that carries a structure of a topological operad; the required change is that one of the spaces $Y_i$ is chosen as the base and is required to sit at the root of each cactus.
\comment{That modification is defined for an arbitrary pointed topological space $(Y,\bullet)$;}
We call the modified space the space of based $Y$-cacti. The goal of this paper is to understand the algebra and topology of this operad.

For homology with coefficients in a field, we show that the homology operad of the operad of based $Y$-cacti is obtained from the homology coalgebra of $Y$ by a formal algebraic procedure that only uses the augmentation and the coproduct; thus, it is defined for every graded cocommutative coalgebra~$C$, not necessarily the homology coalgebra of a topological space. Remarkably, for every coalgebra $C$ this defined operad is Koszul. To prove that, we use filtered distributive laws between operads, as defined by the second author in~\cite{Dotsenko2006}.  One immediate consequence of our results is that, for $Y=S^1$, the homology operad of based $Y$-cacti is isomorphic, as an $\mathbb{S}$-module, to $\Perm\circ\PreLie_1$, \comment{where $\Perm$ is the operad of associative permutative algebras (which are essentially ``based commutative algebras'', that is commutative algebras with a distinguished factor in every product) and $\NAP$ is the operad of nonassociative permutative algebras 
(which is a ``degeneration'' of the operad $\PreLie$ of pre-Lie algebras mentioned above)} 
which, given that the operad of associative permutative algebras $\Perm$ encodes commutative algebras with additional structure, may be naturally thought of as an ``operad-compatible improvement'' of the result of~\cite{JMM06} mentioned above. 

Our constructions are defined over a field of arbitrary characteristic, and our results on operads of based cacti hold in that generality. However, even the distributive law criterion for Koszulness, let alone its filtered generalisation, has only been available in zero characteristic, since the known proofs \cite{Dotsenko2006,Vallette_PROP} rely on the K\"unneth formula for symmetric collections. Using the shuffle operads technique \cite{DKRes,DK}, we were able to obtain a characteristic-independent proof of this criterion. 

The paper is organised as follows. In Section~\ref{sec:background}, we recall necessary background information that we use throughout the paper. In Section~\ref{sec:top_cacti}, we define the topological operads of based cacti and discuss its connections both with automorphism groups of free products and with other known topological operads. The homology operad for the operad of based cacti is computed in Section~\ref{sec:homology}. In Section~\ref{sec:distr-laws}, we discuss filtered distributive laws between quadratic operads. Section~\ref{sec:Koszul} shows how to use filtered distributive laws to prove the Koszul property for the linear operads of based cacti, and also discuss its applications to the operad of post-Lie algebras and the operad of commutative tridendriform algebras.

\section{Trees, coalgebras, operads}\label{sec:background}

All ``linear'' objects in this paper (algebras, coalgebras, operads) will be enriched in a certain symmetric monoidal category~$(\mathcal{C},\otimes,\sigma,\mathbb{I})$, usually the category $\Vect$ of vector spaces or the category $\gVect$ of graded vector spaces (over some field~$\k$; unless otherwise specified, we do not make any assumptions on its characteristic). Whenever appropriate, we assume vector spaces to be finite-dimensional, or possessing an additional $\mathbb{N}$-grading with finite-dimensional homogeneous components; this allows to approach tensor constructions and duals with ease, freely pass between an algebra and its dual coalgebra etc.

\subsection{$\bdY$-labelled trees}
A \emph{tree} is an acyclic connected graph and a \emph{rooted tree} is a tree with a chosen vertex, the \emph{root}.
A rooted tree may be directed: every edge $\cbrac{v,w}$ may be oriented to $\ed{vw}$ in such a way that the minimal path from $w$ to the chosen vertex contains $\cbrac{v,w}$. By the acyclicity of the tree this must hold for exactly one of the choices $\ed{vw}$ and $\ed{wv}$.
The edges may be seen to be directed `away from the roots'.
We denote by $E(T)$ the set of edges of a tree $T$.

Suppose that $T$ is a tree with vertex set $V$.  Let $\bdY=\brac{Y_i}_{i\in V}$ be a $V$-tuple of topological spaces.
Then a \emph{$\bdY$-tree} is a rooted tree with an edge labelling where the edge $\ed{ij}$ is labelled by an element of $Y_i$.
For a space $Y$ as shorthand we define a $Y$-tree to be a $\bdY$-tree where the $V$-tuple $\bdY$ is constantly $Y$.
Then the edge labelling is a map from the edge set $E$ to the space $Y$. Meanwhile a $Y$-forest is a $\bdY$-tree where $\bdY$ is the $V$-tuple with $Y_0\cong \cbrac{\dt}$, where $0$ is the root vertex and $Y_v\cong Y$ for any other vertex. The naming makes sense because by removing the root $0$ and all adjacent vertices we are left with a disjoint union of $Y$-trees; the root of each tree is the unique vertex adjacent to $0$ and the edge labelling is inherited.

To a rooted tree $T$ we define the \emph{level} $l(T)$ to be the number of non-trivial directed paths in $T$.  So for a corolla with root $1$ and $k-1$ other vertices the level is $k-1$, for a tree with root $1$ and edges $\ed{i(i+1)}$ for $i=1\clc k-1$ the level is $k(k-1)/2$.
The level allows one to filter the set of $\bdY$-trees.

\subsection{Coalgebras}
A coalgebra is an object~$C$ of $\mathcal{C}$ equipped with a comultiplication $\Delta\colon C\to C\otimes C$ and a counit $\epsilon\colon C\to\mathbb{I}$ satisfying the conventional coassociativity and counit axioms. For the comultiplication, we often use Sweedler's notation $\Delta(c)=\sum c_{(1)}\otimes c_{(2)}$. An augmented coalgebra is a coalgebra $C$ equipped with a coalgebra homomorphism $\gamma\colon\mathbb{I}\to C$ such that $\epsilon\gamma=1$. A cocommutative coalgebra is a coalgebra satisfying $\sigma\Delta=\Delta$. Our main focus will be on graded augmented cocommutative coalgebras, that is augmented cocommutative coalgebras in $\gVect$. The main source of such coalgebras relevant for our purposes is topology: the homology coalgebra of a pointed topological space~$(Y,\dt)$ is a graded augmented cocommutative coalgebra. An augmented coalgebra in $\Vect$ or $\gVect$ naturally splits into a direct sum of vector spaces $C=\k\one\oplus\overline{C}$, where $\one=\gamma(1)$, $\overline{C}=\ker(\epsilon)$.

\subsection{Operads}

For details on operads we refer the reader to the book~\cite{LodayVallette2011}, for details on Gr\"obner bases for operads~--- to the paper~\cite{DK}. In this section we 
only recall the key notions used throughout the paper. By an operad (enriched in a symmetric monoidal category~$(\mathcal{C},\otimes,\sigma,\mathbb{I})$) we mean a monoid 
in one of the two monoidal categories: the category of symmetric $\mathcal{C}$-collections equipped with the composition product or the category of nonsymmetric 
$\mathcal{C}$-collections equipped with the shuffle composition product. The former kind of monoids is referred to as symmetric operads, the latter~--- as shuffle operads. 
We always assume that our collections are reduced, that is, have no elements of arity~$0$. A good rule of thumb is that all operads defined in this paper are symmetric 
operads, but for computational purposes it is useful to treat them as shuffle operads. This does not lose any information except for the symmetric group actions, since the 
forgetful functor $\calO\mapsto\calO^f$ is monoidal and one-to-one on objects (and therefore for tasks that can be formulated without the symmetric group actions, e.g. 
computing bases and dimensions of components, proving the Koszul property etc., we can choose arbitrarily whether to work with a symmetric operad or with its shuffle 
version). In the ``geometric'' setting, $\mathcal{C}$ will usually be the category of sets, or topological spaces, or pointed topological spaces, in the ``linear'' 
setting~--- the category of vector spaces (in which case symmetric collections are usually called $\mathbb{S}$-modules), or the category of graded vector spaces or chain
complexes (in which 
case symmetric collections are called differential graded $\mathbb{S}$-modules). A linear symmetric operad can also be thought as of collection of spaces of operations of 
some type, and therefore can be defined via its category of algebras, i.e. vector spaces where these operations act, via identities between operations acting on a vector 
space.

In the linear setting, a very useful technical tool for dealing with (shuffle) operads is given by Gr\"obner bases. More precisely, similarly to associative algebras, 
operads can be presented via generators and relations, that is as quotients of free operads $\calF(\calV)$, where $\calV$ is the space of generators. The free shuffle 
operad generated by a given nonsymmetric collection admits a basis of ``tree monomials'' which can be defined combinatorially; a shuffle composition of tree monomials is 
again a tree monomial. In addition to the ``arity'' of elements of a free operad, there is the notion of weight, similar to grading for associative algebras: we define the 
weight of a tree monomial as the number of generators used in this tree monomial. Weight is well behaved under composition: when composing several tree monomials, the 
weight of the result is equal to the sum of their weights. For an arbitrary operad~$\calO=\calF(\calV)/(\calR)$ whose relations $\calR$ are weight-homogeneous, the weight 
descends from the free operad $\calF(\calV)$ on~$\calO$; the subcollection of~$\calO$ consisting of all elements of weight~$k$ is denoted by~$\calO_{(k)}$.

There exist several ways to introduce a total ordering of tree monomials in such a way that the operadic compositions are compatible with that total ordering. There is also a combinatorial definition of divisibility of tree monomials that agrees with the naive operadic definition: one tree monomial occurs as a subtree in another one if and only if the latter can be obtained from the former by operadic compositions. A Gr\"obner basis of an ideal $I$ of the free operad is a system $S$ of generators of~$I$ for which the leading monomial of every element of the ideal is divisible by one of the leading terms of elements of~$S$. Such a system of generators allows to perform ``long division'' modulo~$I$, computing for every element its canonical representative. There exists an algorithmic way to compute a Gr\"obner basis starting from any given system of generators (``Buchberger's algorithm for shuffle operads''). 

A part of the operad theory which provides one of the most useful known tools to study homological and homotopical algebra for algebras over the given operad is the Koszul duality for operads~\cite{GK}. Proving that a given operad is Koszul instantly provides a minimal resolution for this operad, gives a description of the homology theory and, in particular, the deformation theory for algebras over that operad etc. There are a few general methods to prove that an operad is Koszul; one of the simplest and widely applicable methods~\cite{DK,DKRes} is to show that a given operad has a quadratic Gr\"obner basis (as a shuffle operad); this provides a sufficient (but not necessary) condition for Koszulness of an operad. If an operad is Koszul, it necessarily is \emph{quadratic}, that is has weight-homogeneous relations of weight~$2$. 

The operads that serve as ``building blocks'' for operads considered throughout the paper are mostly well known: $\Com$ (commutative associative algebras), $\Lie$ (Lie algebras), $\As$ (associative algebras), $\Leib$ (Leibniz algebras~\cite{Loday93}), $\Zinb$ (Zinbiel algebras~\cite{Loday95}), $\Perm$ ([associative] permutative algebras~\cite{Chapoton2001}), $\NAP$ (nonassociative permutative algebras~\cite{Livernet2006}, closely related to ``right-commutative magma''~\cite{DL2002}). All these operads are Koszul, and have a quadratic Gr\"obner basis.

\subsection{Polynomial functors}

As we said before, some of our constructions exist both in a ``geometric'' and a ``linear'' setting, and are related to each other via the homology functor (which assigns to a topological space $Y$ the graded cocommutative coalgebra~$H_*(Y)$). To make additional structures transfer easily, we use basic concepts of the theory of polynomial functors. A polynomial functor is a notion that categorifies the notion of a polynomial, and more generally of a formal power series. Polynomial functors provide a useful uniform language to deal with categorical constructions that have ``a polynomial flavour'', e.g. when computing sums and products in appropriate categories over specified sets indexing summands/factors in a way that keeps track of the intrinsic structure of the indexing sets. 

In precise words, a diagram of sets and set maps 
\begin{equation}\label{eq:PolyFun}
I\stackrel{s}{\longleftarrow}E\stackrel{p}{\longrightarrow}B\stackrel{t}{\longrightarrow}J
\end{equation}
gives rise to a polynomial functor $F:\Set/I \to \Set/J$ defined by the formula
\begin{equation}\label{eq:PolyFunFormula}
\Set/I \stackrel{s^*}{\longrightarrow} \Set/E \stackrel{p_*}{\longrightarrow}\Set/B\stackrel{t_!}{\longrightarrow} \Set/J.
\end{equation}
Here ${}_*$ and ${}_!$ denote, respectively, the right adjoint and the left adjoint of the pullback functor~${}^*$. More explicitly, the functor is given by
\begin{equation}\label{eq:PolyFunExplicitFormula}
 [f:X\to I] \longmapsto \sum_{b\in B} \prod_{e\in p^{-1}(b)} f^{-1}(s(e)), 
\end{equation}
where the last set is considered to be over $J$ via $t_!$. Here one can replace $\Set$ by another category where all the appropriate notions make sense. For our purposes, it is enough to consider the case $I=J=*$, in which case the corresponding functors were referred to as polynomial functors in \cite{MP2000}, and are called polynomial functors in one variable in more recent literature.

For a systematic introduction to polynomial functors, we refer the reader to the paper~\cite{KJBM2007} and the notes~\cite{Kock2010} that reflect the state-of-art of the theory. 

\section{The operad of cacti}\label{sec:top_cacti}

\subsection{The operad $\NAPY$}
Let $Y$ be a set and let $\NAPY(n)$ be the set of $Y$-trees with vertex set $\sbrac{n}=\cbrac{1\clc n}$.
When $Y$ is a singleton set this is just the set of rooted trees which we denote $\RT(n)$.
The symmetric group $\symn$ acts on $\NAPY(n)$ by permuting elements of the vertex set.
For a given rooted tree the set of $Y$-labellings is equal to $\Hom(E,Y)=Y^E$.
Since the number of edges of a tree on $\cbrac{n}$ is always $n-1$, the set of $Y$-labellings is in turn isomorphic to $Y^{n-1}$.
Hence
\begin{equation}\label{eq_NAPYpoly}
\NAPY(n) \cong \coprod_{T\in\RT(n)} Y^{n-1}.
\end{equation}
In this way if $Y$ is a topological space then we may also apply a topology to $\NAPY(n)$ using the product topology on $Y^{n-1}$.

Now let $T_1\in \NAPY(n)$ and $T_2\in\NAPY(m)$ and $i\in\sbrac{n}$.
We may define a composition $T_1\circ_i T_2\in \NAPY(n+m-1)$ by first identifying the root of $T_2$ with the vertex $i$ in $T_1$.  
This is a tree and may be rooted by taking the root of $T_1$.
The edge set is equal to the union $E(T_1)\amalg E(T_2)$ of the edge sets of $T_1$ and $T_2$ and so one inherits an edge labelling by elements of $Y$.
It has the vertex set
\begin{equation}
\cbrac{1\clc i-1}\amalg\cbrac{1\clc m}\amalg\cbrac{i+1\clc n}.
\end{equation}
We then relabel the vertices by elements of $\sbrac{n+m-1}$ using the isomorphism which fixes $\cbrac{1\clc i-1}$, shifts the set $\cbrac{1\clc m}$ to $\cbrac{i\clc m+i-1}$ and shifts $\cbrac{i+1\clc n}$ to $\cbrac{m+i\clc m+n-1}$.
This gives a rooted $Y$-tree on the vertex set $\cbrac{1\clc n+m-1}$ and so an element $T_1\circ_i T_2\in\NAPY(n+m-1)$.

\begin{proposition}\label{prop:NAPoperad}
Let $Y$ be a set, then the maps
\begin{equation}
\circ_i:\NAPY(n)\times \NAPY(m) \rightarrow \NAPY(n+m-1)
\end{equation}
for $i=1\clc n$ give the collection $\NAPY$ an operad structure.
The operad is generated by its binary operations:
\begin{equation}
\smtree{2\\1\ar[u]_y}\quad\text{ and }\quad\smtree{1\\2\ar[u]_z}
\end{equation}
for $y,z\in Y$ and these satisfy the quadratic relation
\begin{equation}\label{eq_NAPrel}
\smtree{2\\1\ar[u]_y}\circ_1\smtree{2\\1\ar[u]_z}\quad=\quad \brac{\smtree{2\\1\ar[u]_z}\circ_1\smtree{2\\1\ar[u]_y}}.(23).
\end{equation}
\comment{  
If $Y$ is also equipped with a topology then the maps $\circ_i$ are continuous and so $\NAPY$ is a topological operad.
}
\end{proposition}
\begin{proof}
Let $T_1$, $T_2$ and $T_3$ be $Y$-trees in $\NAP_Y(n_1)$, $\NAP_Y(n_2)$ and $\NAP_Y(n_3)$ respectively.
Let $i < j \in \sbrac{n_1}$ and $k\in\sbrac{n_2}$; we must show that the two associativity relations hold;
\begin{equation}\label{eq_Vassociativity}
(T_1\circ_j T_2) \circ_i T_3 = (T_1 \circ_i T_3) \circ_{j + n_3 - 1} T_2
\end{equation}
and
\begin{equation}\label{eq_Lassociativity}
T_1\circ_i (T_2 \circ_k T_3) = (T_1 \circ_i T_2) \circ_{k + i - 1} T_3.
\end{equation}
In both cases we are gluing together trees by identifying vertices~--- in the first we identify the roots of $T_2$ and $T_3$ with the vertices $j$ and $i$ of $T_1$ respectively~--- whilst in the second the root of $T_2$ is joined to vertex $i$ of $T_1$ and the root of $T_3$ is identified with vertex $k$ of $T_2$.
The only complication is that when two trees are composed their vertices are renumbered: this change is taken into account in the right hand side of each equation.
In both cases the edge set of the resulting tree is the union of the edge sets of the three component trees, hence the $Y$-labellings on both sides of each equation are equal. 
It remains to make the routine check that the vertex labels in each side of each equation agree, this is no more complicated than the analogous check in the associative operad.

Now we show that the operad is generated by operations of arity $2$.
Let $T\in\NAP_Y(n)$ be any $Y$-tree and let $\ed{ij}$ be a leaf of $T$; let $y$ be the label of $\ed{ij}$.
By applying a permutation if necessary we may assume that $i=n-1$ and $j=n$.
Letting $T'$ be the $Y$-tree in $\NAP_Y(n-1)$ given by removing the edge $\ed{(n-1)n}$ and the vertex $n$, we have that
\begin{equation}
T = T' \circ_{n-1} \smtree{2\\1\ar[u]_y}.
\end{equation}
Therefore any $Y$-tree may be written as compositions of trees with two vertices and a permutation and so $\NAP_Y$ is generated in arity $2$.

The relation~\eqref{eq_NAPrel} is to seen to hold by evaluating each side of the equation to find the same $Y$-tree
\begin{equation}
\smtree{2 && 3 \\ &1\ar[ur]_y\ar[ul]^z&}.
\end{equation}
\end{proof}
The above theorem gives quadratic relations in the binary generators, the Corollary~\ref{cor:NAPYquadratic} will show that these suffice to present the operad.

\begin{remark}\label{rem_polynomial}
The operads $\NAPY$ are functorial in sets $Y$, in fact $\NAP_{(-)}(n)$ is a polynomial functor given by the diagram
\begin{equation}\label{eq_NAPpolynomial}
\ast \leftarrow \coprod_{T\in\RT(n)} E(T) \rightarrow \RT(n) \rightarrow \ast.
\end{equation}
Both the operad maps and the proof above work on the level of the polynomial itself, hence for any appropriate category one may use the polynomial to give a family of operads $\NAP_{(-)}$.
For instance this means that if $Y$ is also equipped with a topology then $\NAP_Y$ is a topological operad.
In Section~\ref{sec:homology} we will consider the operads $\NAP_D$ where $D$ is a graded vector space.
\end{remark}
\begin{remark}
When $Y$ is a single point $\cbrac{\dt}$, the operad $\NAPY$ is the usual operad~$\NAP$.
\end{remark}

Let us finish this section with a few words on $\NAP_Y$-algebras. One convenient way to think of them is via the ``right regular module'', since the defining relations say that all the right multiplications 
\begin{equation}
R(y,b)\colon a\mapsto \smtree{2\\1\ar[u]^y}(a,b) 
\end{equation}
commute with each other. Somewhat more precisely, let $A$ be an object in a symmetric monoidal category~$\mathcal{C}$, and let
\begin{equation}
f\colon Y\times A\to\Hom_{\mathcal{C}}(A,A) 
\end{equation}
be a map whose image is an abelian submonoid. Then $A$ is a $\NAP_Y$-algebra enriched in~$\mathcal{C}$ with the structure maps given by 
\begin{equation}\label{eq:NAP-via-family}
\smtree{2\\1\ar[u]^y}(a,b)=f(y,b).a.
\end{equation} 
This way to approach $\NAP_Y$-algebras gives a source of examples based on $\Perm$-algebras with a family of maps as follows.

\begin{example}\label{ex:NAP-from-Perm}
Let $(A,\cdot)$ be a $\Perm$-algebra encriched in a symmetric monoidal category $\mathcal{C}$, and let $g_y$, $y\in Y$ be a family of maps in $\Hom_{\mathcal{C}}(A,A)$ (note that these maps
may be arbitrary, not necessarily algebra homomorphisms). Then $A$ is a $\NAP_Y$-algebra enriched in~$\mathcal{C}$ with the structure maps given by 
\begin{equation}\label{eq:NAP-via-hom}
\smtree{2\\1\ar[u]^y}(a,b)=a\cdot g_y(b). 
\end{equation} 
\end{example}

One more observation we want to mention in this section is that the construction of the free $\NAP$-algebra mentioned in~\cite{Livernet2006} admits an immediate generalisation to the case of $\NAP_Y$-algebras: the free
$\NAP_Y$-algebra enriched in $\Set$ with the generating set~$V$ admits a realisation as the set of $Y$-trees whose vertices carry labels from~$V$, with the product defined in the same way as we defined the composition in the operad:
\begin{equation}\label{eq:free-NAP}
\smtree{2\\1\ar[u]^y}(a,b)=(\smtree{2\\1\ar[u]^y}\circ_1 a)\circ_2 b.   
\end{equation}
In this composition the root of $b$ is joined to the root of $a$ by an edge labelled by $y$; the new root is taken to be the root of $a$.

\subsection{The operad of based cacti}
Let $V$ be a set and $\bdY$ be a $V$-tuple of pointed spaces.  Let $T$ be a $\bdY$-tree with root $r\in V$ and suppose that $\ed{ij}$ is an edge of $T$ where $i\neq r$.
Suppose further that $\ed{ij}$ is labelled by the basepoint $\dt\in Y_i$.
Then we say that $\ed{ij}$ is a \emph{reducible edge} and that $T$ is \emph{reducible}.
Since $i$ is not the root there is a unique incoming edge $\ed{ki}$ which is labelled by some $y\in Y_k$.
We define $T_{ij}$ to be the $\bdY$-tree given by removing the edge $\ed{ij}$ and adding the edge $\ed{kj}$ with the label $y\in Y_k$.
We say that $T_{ij}$ is a \emph{reduction} of $T$.
\begin{definition}
Let $V$ be a finite set and $\bdY$ be a $V$-tuple of pointed spaces.
Then the space of \emph{based $\bdY$-cacti}, $\BCact_\bdY$ is the topological space given by quotienting out by the relation $T\sim T_{ij}$ for any $T$ with a reducible edge $\ed{ij}$.

Now let $V_0$ be the set $V\cup\cbrac{0}$ and let $\bdY_0$ be the $V_0$-tuple given by adjoining $Y_0=\cbrac{\dt}$ to the $V$-tuple $\bdY$.
Then we define the space of $\bdY$-cacti, $\Cact_\bdY$ to be the subspace of $\BCact_{\bdY_0}$ consisting of the trees with root $0$.
\end{definition}
\begin{remark}
For each $\bdY$-cactus $T\in \BCact_\bdY$ one may define the space
\begin{equation}\label{eq_cactusproduct}
\bdY(T) = \dfrac{\coprod_{v\in V} Y_v}{y_{ij} \sim \dt_j \mid \ed{ij}\in E(T)},
\end{equation}
where $y_{ij}\in Y_i$ is the label of the edge $\ed{ij}$ and $\dt_j$ is the basepoint of $Y_j$.
Note that this realisation is invariant across equivalences $T\sim T_{ij}$.
If each space $Y_i$ is path connected then this space is homotopy equivalent to the wedge product of the spaces $Y_v$ for $v\in V$.  
These spaces are called \emph{cactus products} and were studied by the second author in~\cite{Griffin}.
There it was shown that the space $\Cact_\bdY$ of such products has interesting homotopical properties, in particular if the spaces $Y_i$ are classifying spaces for groups $G_i$ then $\Cact_\bdY$ is a classifying space for the Fouxe-Rabinovitch group $\smallcaps{FR}(G)$ of partial conjugation automorphisms of the free product $G=\ast_{i\in V}G_i$.
An example of a cactus product:
\begin{equation}
\vcenter{\hbox{\includegraphics[height=25mm]{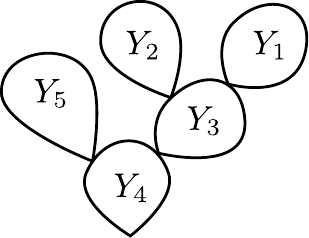}}}
\end{equation}
Note that if $v$ is the root of the tree $T$ then the space $Y_v$ must always be at the `base' of the diagram.
The appearance of the diagram explains the term `based $\bdY$-cactus'.
We also see the reason for adjoining a point space $Y_0$; this removes the base space; the space $Y_0$ acts as a basepoint.
\end{remark}
\begin{remark}
Recall that the level of a rooted tree is the number of non-trivial directed paths.
When $\ed{ki}$ and $\ed{ij}$ are edges of a rooted tree $T$, the rooted tree $T'$ given by removing $\ed{ij}$ and then adding $\ed{kj}$ has strictly lesser level.  
Indeed if $P$ is the unique path joining vertices $v$ and $w$ in $T'$, then there is a unique path joining $v$ and $w$ in $T$.  
But the number of paths in $T$ is strictly larger because there is a path joining $i$ and $j$ in $T$ but not in $T'$.
So for any $\bdY$-tree $T$ one may use the reductions $T\sim T_{ij}$ repeatly until there are no reducible edges remaining.
Since the level reduces each time this process must terminate.
It is easy to check that it does not matter what order the reductions $T\sim T_{ij}$ are applied because if $\ed{ab}$ and $\ed{cd}$ are two reducible edges then $(T_{ab})_{cd}=(T_{cd})_{ab}$.
Hence for each $Y$-labelled tree there is a unique equivalent tree which can not be reduced any further.
Therefore $\BCact_\bdY$ is isomorphic to the set of irreducible $\bdY$-trees.
\end{remark}
\begin{definition}
Let $(Y,\dt)$ be a pointed space.
For $n\geq 1$, we define the space $\BCY(n)$ to be the space of based cacti on the $n$-tuple $\bdY=\brac{Y_i\cong Y}_{i=1\clc n}$.
The action of $\symn$ on $\cbrac{1\clc n}$ makes this into a symmetric collection.
\end{definition}

\begin{theorem}\label{thm:BCYoperad}
Let $(Y,\dt)$ be a pointed space.
The equivalence relation $\sim$ generated by reductions $T\sim T_{ij}$ is compatible with the operad maps of $\NAPY$.
Hence the quotient collection $\BCY$ has an operad structure inherited from $\NAPY$.
Furthermore the equivalence relation $\sim$ is generated as an operad ideal by the single relation
\begin{equation}\label{eq_cactrel}
\smtree{3 \\ 2\ar[u]_\dt\\ 1\ar[u]_y} \quad=\quad \smtree{2 && 3 \\ &1\ar[ul]^y \ar[ur]_y& }.
\end{equation}
\end{theorem}
\begin{proof}
Let $T\in\NAP_Y(n)$ be a $Y$-tree with reducible edge $\ed{ij}$; let $T'\in\NAP_Y(m)$ be any other $Y$-tree.
Then for any $k\in\sbrac{n}$ and $l\in\sbrac{m}$ the products
\begin{equation}
T\circ_k T'\quad\text{ and }\quad T'\circ_l T
\end{equation}
are both given by identifying vertices.  The edge $\ed{ij}$ still exists in each product although it may have been relabelled, to $\ed{i'j'}$ say.
The label in $Y$ is still the point~$\dt$.  Furthermore $i'$ is not the root in either product so $\ed{i'j'}$ is a reducible edge giving the reductions
\begin{equation}
(T\circ_k T') \sim (T\circ_k T')_{i'j'}\quad\text{ and }\quad (T'\circ_l T) \sim (T'\circ_l T)_{i'j'}.
\end{equation}
The reductions are also closed under the symmetric actions: for $\sigma\in\symn$ the edge $\ed{(i\sigma)(j\sigma)}$ is reducible in $T\sigma$.
This shows the first part and in particular that $\BCY$ is an operad.

We will now show that all reductions $T\sim T_{ij}$ are obtainable from the reduction~\eqref{eq_cactrel} of 
\begin{equation}\label{eq_123tree}
\smtree{3 \\ 2\ar[u]_\dt \\ 1\ar[u]_y}.
\end{equation}
We must show that any reducible $Y$-tree $T$, with reducible edge $\ed{ij}$ say, is contained in the ideal in $\NAP_Y$ generated by~\eqref{eq_123tree}.
Let $\ed{ki}$ be the unique edge incoming to $i$.  By applying a permutation we may assume that $k=1$, $i=2$ and $j=3$.
The essential idea of the proof is that since~\eqref{eq_123tree} is a subtree, the tree $T$ may be written as a composition of~\eqref{eq_123tree} and other $Y$-trees.
Removing the edges $\ed{12}$ and $\ed{23}$ from $T$ leaves three connected components; $T_1$ contains $1$, $T_2$ contains $2$ and $T_3$ contains $3$.
In effect we have partitioned the edge set of $T$ into $\{\ed{12},\ed{23}\}$, $E(T_1)$, $E(T_2)$ and $E(T_3)$.
Then we may express $T$ as
\begin{equation}
T = \Bigl(\Bigl(\Bigl( T_1 \circ_1 (\xymatrix@C=0.8pc{1\ar[r]^y & 2\ar[r]^\dt & 3})\Bigr) \circ_3 T_3\Bigr) \circ_2 T_2\Bigr) . \sigma,
\end{equation}
where $\sigma$ is a permutation relabelling the vertices.
\end{proof}
\begin{remark}
The Corollary~\ref{cor:NAPYquadratic} to Theorem~\ref{thm:NAP-Koszul} states that $\NAP_Y$ is binary quadratic.
Along with the Theorem above this shows that $\BCY$ is also binary quadratic.
\end{remark}

In the spirit of how we approached $\NAP_Y$-algebras, a $\BCact_Y$-algebra enriched in a symmetric monoidal category $\mathcal{C}$ is a $\NAP_Y$-algebra enriched in~$\mathcal{C}$ where the operation $\smtree{2\\1\ar[u]_\dt}$ 
is associative, and 
\begin{equation}\label{eq:BCact-via-family}
f(y,\smtree{2\\1\ar[u]_\dt}(a,b))=f(y,a)\circ f(y,b). 
\end{equation}

\subsection{The fundamental groupoid of $\BCY$}
Let $Y$ be a topological space and let $P$ be a subset of $Y$.
We define the fundamental groupoid $\pi_1(Y,P)$ to be the groupoid with objects the points $p\in P$ and morphisms the homotopy classes of paths in $Y$ which start and end in elements of $P$.
The composition is by concatenation of paths and the units are supplied by the constant paths.
So if $(Y,\dt)$ is a pointed space then $\pi_1(Y,\cbrac{\dt})$ is the fundamental group of $Y$.
Let $(Y,\dt)$ be a pointed space and let $P\in Y$ be a set of points which contains $\dt$ and such that each path connected component of $Y$ contains a single point of $P$.  This may be seen as a section of the map
\begin{equation}
(Y,\dt) \rightarrow (\pi_0(Y),\pi_0(\dt)).
\end{equation}
Then by the functoriality of $\BCact_{(-)}$ there is a pair of operad maps
\begin{equation}
\xymatrix{\BCY \ar@<0.5ex>[r] & \BCact_P, \ar@<0.5ex>[l]}
\end{equation}
which serves to pick out a single element in each path connected component of $\BCY$.
The fundamental groupoid functor preserves products and colimits and so $\pi_1(\BCY;\BCact_P)$ is an operad in the category of groupoids.

From now on we will restrict $Y$ to be a path connected space, so $P=\cbrac{\dt}$.
In this case $\BCact_P \cong \Perm$, the operad for permutative algebras~--- each of the $n$ elements is given by a corolla.
So we see that $\BCY(n)$ is made up of $n$ components and the action of $\symn$ gives isomorphisms between them.
Denote by $\BCact_Y(n)_r$ the component consisting of trees with root $r$.

\begin{proposition}\label{prop:BCYgroupoid}
The fundamental group of $\BCact_Y(n)_1$ is presented by generators $\alpha^g_{ij}$ for $i=1\clc n$, $j=2\clc n$ with $i\neq j$ and $g\in\pi_1(Y,\dt)$, along with relations
\begin{equation}\label{eq_BCGrel1}
\alpha_{ij}^g\alpha_{ij}^h = \alpha_{ij}^{gh},
\end{equation}
\begin{equation}\label{eq_BCGrel2}
\sbrac{\alpha^g_{ij},\alpha^h_{ik}}=e
\end{equation}
for distinct $i$, $j$, $k$;
\begin{equation}\label{eq_BCGrel3}
\sbrac{\alpha^g_{ij},\alpha^h_{kl}}=e
\end{equation}
for distinct $i$, $j$, $k$, $l$; and
\begin{equation}\label{eq_BCGrel4}
\sbrac{\alpha^g_{ij}\alpha^g_{ik},\alpha^h_{jk}}=e
\end{equation}
for distinct $i$, $j$, $k$.
\end{proposition}
\begin{proof}
We defined the cactus operads $\BCY$ by adding certain relations $T\sim T_{ij}$ for trees $T$ with a reducible edge $\ed{ij}$.
The relations come in families: for a fixed tree $T$ with a fixed edge $\ed{ij}$ where $i$ is not the root, a $Y$-tree is reducible if $\ed{ij}$ is labelled by the point $\dt\in Y$ and the remaining $n-2$ edges are labelled by any element in $Y$, so there is a family of relations parametrised by $\cbrac{\dt}\times Y^{n-2}$.
Each element in this family encodes a reduction $T\sim T_{ij}$:  there is one map from $Y^{n-2}$ corresponding to $T$ and another map from $Y^{n-2}$ corresponding to $T_{ij}$.  
For the second map the diagonal $y\mapsto (y,y)$ is used to define the new labelling.
So for each such tree $T$ with edge $\ed{ij}$ there are a pair of maps
\begin{equation}\xymatrix{
Y^{n-2}\ar@<0.5ex>[r]\ar@<-.5ex>[r] & \NAP_Y(n).}
\end{equation}
In identifying the two images of each point we are taking the coequaliser of this diagram.
But we have such an identification for each tree $T$ with an edge $\ed{ij}$ where $i$ is not the root.
So we have a diagram with a copy of $Y^{n-2}$ for each such pair $(T,\ed{ij})$ and two arrows from each copy to a single copy of $\NAP_Y(n)$.
The colimit of this diagram is the space given by making all identifications $T\sim T_{ij}$ -- that is, the colimit is $\BCY(n)$.

We will use $G$ to denote the group $\pi_1(Y,P)$.
The fundamental groupoid functor $\pi_1$ respects colimits and products, so in particular respects polynomial functors meaning that $\pi_1(\NAP_Y,\NAP_P) \cong \NAP_G$.
Furthermore 
\begin{equation}
\BCact_G(n) := \pi_1(\BCY(n),\BCact_P(n))
\end{equation}
is given by the colimit of the diagram which consists of a single copy of $\NAP_{G}(n)$ and a copy of $G^{n-2}$ for each pair $(T,\ed{ij})$.
It now remains to compute this colimit.

Restricting ourselves to trees with root $1$, we have that $\BCact_Y(n)_1$ is the colimit of the diagram where $\NAP_Y(n)$ is replaced by $\NAP_Y(n)_1$ and we only include pairs $(T,\ed{ij})$ where the root of $T$ is 1.
Since $\BCact_\dt(n)_1$ is a single point, $\BCY(n)_1$ is connected and $\BCact_{G}(n)_1$ has a single object and so may be viewed as a group.

We will now examine the effect of coequalisers on morphisms.  A generic morphism of $\NAP_G(n)_1$ consists of a rooted tree $T\in\RT(n)_1$ with edge labels $g_e\in G$ for each $e\in E(T)$.
But since such elements belong to a component of $\NAP_G(n)$ isomorphic to $G^{n-1}$ they can be rewritten as the product of $n-1$ elements, one for each edge.
The element corresponding to $e\in E(T)$ is given by labelling edge $e$ by $g_e$ and every other edge of $T$ by the identity.
We will denote such an element by $g_{(T,\ed{ij})}$, which is the tree $T$ with edge $\ed{ij}$ labelled by $g\in G$.

The coequalisers encode reductions just as before.  Let $g_{(T,\ed{vw})}\in\NAP_G(n)_1$ be a generator where $g\in G$, $T\in\RT(n)_1$ and $\ed{vw}\in E(T)$ is any edge.
Let $\ed{ij}$ be another edge, this time we ask that $i$ is not the root $1$; this will be the edge we will reduce over.
As before let $\ed{ki}$ be the unique incoming edge to $i$ and let $T_{ij}$ be the tree given by cutting $\ed{ij}$ and adding $\ed{kj}$.
The element $g_{(T,\ed{vw})}$ may be reduced when the label of $\ed{ij}$ is the identity, that is if $\ed{vw}\neq\ed{ij}$: in the case that $\ed{vw}=\ed{ki}$ the reduced tree has edges $\ed{ki}$ and $\ed{kj}$ labelled by $g$ and the remaining edges labelled by the identity.
In all other cases the single edge $\ed{ki}$ is labelled by $g$ with the remaining edges labelled by the identity.
So if $\ed{vw}=\ed{ki}$ we have $g_{(T,\ed{ki})} = g_{(T_{ij},\ed{ki})}.g_{(T_{ij},\ed{kj})}$ and otherwise $g_{(T,\ed{ki})} = g_{(T_{ij},\ed{ki})}$.
Remember that $\ed{ij}$ must be a reducible edge.

The reductions above allow (using the fact that reduction reduces the level) any element in $\BCact_G(n)$ to be written as a product of elements $g_{(T,\ed{ij})}$ where $T$ is a tree with no identity labelled reducible edges.
The possibilities are that $T$ is a corolla and so $i=1$, or that $T$ is the tree with $n-2$ edges emanating from the root $1$ and the only other edge being $\ed{ij}$.
So for each pair $(i,j)$ where $i,j\in\sbrac{n}$, $i\neq j$ and $j\neq 1$, there is a unique tree $T$ such that the pair $(T,\ed{ij})$ is irreducible.
Therefore we may denote the element $g_{(T,\ed{ij})}$ by $\alpha^g_{ij}$ and these elements generate the group $\BCact_G(n)_1$.

Let $T\in\RT(n)$, $\ed{ij}$ be any edge and $g\in G$, we may write the element $g_{(T,\ed{ij})}$ as a monomial in the generators above as follows.
Let $A_{ij}$ be the set of vertices $v$ which may be joined by a directed path from $i$ to $v$ starting in the edge $\ed{ij}$.  Then $g_{(T,\ed{ij})}$ reduces to the product
\begin{equation}\label{eq_fullreduction}
\prod_{v\in A_{ij}} \alpha^g_{iv}.
\end{equation}

It remains to find the relations between the generators.  
Some of the relations are contributed by the components of $\NAP_G(n)_1$ corresponding to the irreducible pairs $(T,\ed{ij})$.  
The relations $\alpha^g_{ij}.\alpha^h_{ij} = \alpha^{gh}_{ij}$ come from their respective components, these account for~\eqref{eq_BCGrel1}.
Then there are the relations $[\alpha^g_{1i},\alpha^h_{1j}]$ for $i\neq j$, which exist in the component of $\NAP_G(n)_1$ corresponding to the corolla, these account for some of the relations in~\eqref{eq_BCGrel2}, specifically the relations for $i=1$.
Denote by $T(ij)$ the tree with edges $\ed{1k}$ for $k\neq j$ and edge $\ed{ij}$, then this contributes the relations
\begin{equation}
\sbrac{\alpha^g_{ij}, h_{(T(ij),\ed{1k})}}=e.
\end{equation}
But the second element is reducible: in the case $k\neq i$ it reduces to $\alpha^h_{1k}$, whilst in the case $k=i$ it reduces to $\alpha^h_{1i}.\alpha^h_{1j}$.

However these are not all of the relations, additional commutation relations come from other trees $T$.
Let $T(ij,ik)$ be the tree with the edges $\ed{ij}$ and $\ed{ik}$ and edges $\ed{1l}$ for $l\neq j,k$.
This tree encodes commutator brackets
\begin{equation}
\sbrac{g_{(T(ij,ik),\ed{ij})}, h_{(T(ij,ik),\ed{ik})}}=e,
\end{equation}
the elements reduce to $\alpha^g_{ij}$ and $\alpha^h_{ik}$ respectively.
Similarly for distinct $i,j,k,l\neq 1$ let $T({ij},{kl})$ be the tree with edges $\ed{ij}$ and $\ed{kl}$ and edges $\ed{1m}$ for $m\neq j,l$; as above this encodes a commutator relation:
\begin{equation}
\sbrac{g_{(T(ij,kl),\ed{ij})}, h_{(T(ij,kl),\ed{kl})}}=\sbrac{\alpha^g_{ij},\alpha^h_{kl}} = e.
\end{equation}
Finally let $T(ij,jk)$ be the tree with edges $\ed{ij}$ and $\ed{jk}$ and edges $\ed{1l}$ for $l\neq j,k$.
This tree gives the commutator relations
\begin{equation}
\sbrac{g_{(T(ij,jk),\ed{ij})}, h_{(T(ij,jk),\ed{jk})}}=e.
\end{equation}
The second element reduces to $\alpha^h_{jk}$ and the first to the product $\alpha^g_{ij}.\alpha^g_{ik}$.
This accounts for all of the relations in the statement of the proposition.

To show that the stated relations are sufficient to present the group we need to show that the commutator relations
\begin{equation}
\sbrac{g_{(T,\ed{ij})}, h_{(T,\ed{kl})}}=e
\end{equation}
hold for each tree $T$ and each pair of edges $\ed{ij},\ed{kl}$.
Let $A_{ij}$ and $A_{kl}$ be the sets of vertices which index the respective decompositions of the form~\eqref{eq_fullreduction}.
If $A_{ij}$ and $A_{kl}$ are disjoint then commutator relations of the form~\eqref{eq_BCGrel2} and~\eqref{eq_BCGrel3} show that all the constituent irreducible elements commute with one another.
In the case that $A_{ij}$ and $A_{kl}$ do intersect there must be either a directed path from $i$ to $k$ or from $k$ to $i$.
Assuming the former we find that $A_{ij}$ contains $A_{kl}$.
We now show that each $\alpha^h_{kv}$ for $v\in A_{kl}$ commutes with $g_{(T,\ed{ij})}$.
Since both $k$ and $v$ are in $A_{ij}$ the element $\alpha^g_{ik}.\alpha^g_{iv}$ is a term in $g_{(T,\ed{ij})}$, the relation~\eqref{eq_BCGrel4} means that $\alpha^h_{kv}$ commutes with this term.
The remaining terms are of the form $\alpha^g_{iw}$ for $w\neq k,v$ which also commutes with $\alpha^h_{kv}$.  Therefore $\alpha^h_{kv}$ commutes with the element $g_{(T,\ed{ij})}$; and therefore $h_{(T,\ed{kl})}$ commutes with it as well.

Therefore the relations~\eqref{eq_BCGrel1}-\eqref{eq_BCGrel4} suffice to present $\BCact_G(n)_1$.
\end{proof}
We have already seen that $\pi_1(\BCY,\BCact_P)$ is an operad, to give the composition maps we need only describe the compositions on the generating morphisms.
In fact since we have $g\circ_i h = (g\circ_i e).(e\circ_i h)$ we need only describe the compositions of generators with identity maps.
\begin{proposition}
Let $(Y,\dt)$ be a path connected pointed space and let $G$ be its fundamental group.
The operad structure on $\pi_1(\BCY,\BCact_P)$ is given on generating morphisms as follows: let $\alpha^g_{ij} \in\BCact_G(n)_r$ and $e\in\BCact_G(m)_s$ be the identity morphism.
For $a\in\sbrac{m}$ define $i'= i + a - 1$ and $j'=j + a - 1$, then we have
\begin{equation}
e\circ_a\alpha^g_{ij} = \alpha^g_{i'j'}.
\end{equation}
For $b\in\sbrac{n}$ define $i''$ to be $i$ if $i < b$, to be $i+m-1$ if $i>b$ and $i+s-1$ if $i=b$; define $j''$ similarly.
For each $l\in\sbrac{m}$ define $l''$ to be $l + b - 1$.
Then we have
\begin{equation}
\alpha^g_{ij}\circ_b e = \begin{cases}  
\prod_{l=1}^m \alpha^g_{i''l''} & \text{ if $b=j$, and }\\
\alpha^g_{i''j''} & \text{ otherwise. }
\end{cases}
\end{equation}
\end{proposition}
\begin{proof}
Let $T(ij)_r$ be the tree with root $r$, the edge $\ed{ij}$ and $(n-2)$ edges $\ed{rk}$ (if $i=r$ then this is a corolla).
Then $\alpha^g_{ij}$ is represented by the tree $T(ij)_r$ with $\ed{ij}$ labelled by $g\in G$.
Let $C_s$ be the corolla with root $s$ and $m-1$ edges $\ed{sk}$.
When all of the edges are labelled by the identity $e\in G$ then this represents the identity $e$ of $\BCact_G(m)_s$.

To compute $e \circ_a \alpha^g_{ij}$ we compose trees to get $C_s \circ_a T(ij)_r$ and then reduce using Equation~\eqref{eq_fullreduction}.
The unique labelled edge $\ed{ij}$ of $T(ij)_r$ is a leaf and hence it is also a leaf of $C_s \circ_a T(ij)_r$, although now the edge is $\ed{i'j'}$.
Since it is a leaf it reduces to $\alpha^g_{i'j'}$ as required.

To compute $\alpha^g \circ_b e$ is a little more complicated as it depends on the value of $b$.
If $b\neq j$ then the leaf $\ed{ij}$ is still a leaf of $T(ij)_r\circ_b C_s$ and so the same argument applies to give the reduction to $\alpha^g_{i''j''}$.
However if $b=j$ then the tree consists of the edge $\ed{i''j''}$, another $n-2$ edges emanating from the root and $m-1$ edges $\ed{j''l''}$.  
The only labelled edge is $\ed{i''j''}$ and the set $A_{ij}$ of vertices `above $i$' consists of the vertex $j''=s''$ and the vertices $l''$ for each edge $\ed{jl}\in C_s$.
An application of Equation~\eqref{eq_fullreduction} serves to finish the proof.
\end{proof}
\begin{remark}
The groups $\BCact_G(n)_r$ act faithfully on the free product $G^{\ast n}$.  We will write this free product as $G_1\ast\ldots\ast G_n$ where each group is isomorphic to $G$ in order to distinguish between different factors. 
The element $\alpha_{ij}^g$ acts on the factors as follows
\begin{equation}
\alpha_{ij}^g (h) = \begin{cases}  h^{g_i^{-1}} & \text{ if $h\in G_j$ and where $g_i=g$ in $G_i$ and} \\
h & \text{ if $h\in G_k$ for $k\neq j$.}
\end{cases}
\end{equation}
In~\cite{Griffin} the closely related spaces of unbased cacti $\Cact_\bdY$ were studied and it was shown that when $Y_i$ is a classifying space for $G_i$ then $\Cact_\bdY$ is itself a classifying space for a certain group of automorphisms.
As a consequence of Theorems~\ref{thm:homology-BCY} and~\ref{thm:cacti-koszul} we see that
\begin{equation}
H_\ast(\BCY) \cong \Perm\circ \NAP_{\overline{H}(Y)},
\end{equation}
whereas in~\cite{Griffin} it is shown that 
\begin{equation}
H_\ast(\Cact_Y) \cong \Com\circ \NAP_{\overline{H}(Y)}.
\end{equation}
This last isomorphism could also be shown using the methods of reduction used in this paper, although $\Cact_Y$ is not an operad.
\end{remark}

\subsection{Relationships with other topological operads}
The pure braid group on $n$ strands, $P_n$ is known to be a subgroup of the group $P\Sigma_n\cong\pi_1(\Cact_Y(n))$ of partial conjugations of the free group on $n$ letters. This inclusion may be realised by a construction involving cacti. In~\cite{Kaufmann05} various (quasi-)operads of cacti are discussed; these are different to the operad $\BCact_{S^1}$ in that the cacti are planar and unbased.  We will take $\PLC$ to be the spineless and normalised varieties of cacti from~\cite{Kaufmann05}. This quasi-operad is quasi-isomorphic to the little discs operad and so in particular the fundamental group $\pi_1(\PLC(n))$ is the pure braid group $P_n$. There is an $\symn$-equivariant map 
\begin{equation}
\PLC(n) \rightarrow \Cact_{S^1}(n)
\end{equation} 
defined by the map which forgets the planar structure of a planar cactus leaving a cactus product of circles as defined in~\eqref{eq_cactusproduct}; on fundamental groups this gives the inclusion $P_n\rightarrow P\Sigma_n$.
The operad compositions of $\BCact_{S^1}$ and $\PLC$ are not closely related, this may be seen by examining the homology operads which are $\BCact_{H_*(S^1)}$ as defined in the next section and the Gerstenhaber operad $e_2$.

However both families of cacti are related by a third operad which `contains' both. Let $\LR(n)$ be the space of smooth, disjoint embeddings of $n$ copies of the filled in torus, or \emph{ring} $R=S^1\times D^2$ into itself~--- this is naturally an operad. The little discs operad consists of disjoint embeddings of copies of a disc $D^2$ into itself and can be mapped into the little rings operad $\LR$ by applying $\text{id}_{S^1}\times (-)$ to the embeddings.
The image of the little discs operad involves little rings which wind around the large ring once. Meanwhile the operad $\BCact_{S^1}$ is related to the connected components of embeddings in which one little ring, the \emph{root} winds around the large ring once; the remaining rings do not wind around the large ring and all of the rings are unknotted and unlinked. The fundamental groups of these connected components contain $\pi_1(\BCact_{S^1})\cong \BCact_\mathbb{Z}$ as a suboperad.
There are additional elements not in the suboperad given by little rings circling through the large ring along with smooth endomorphisms of $R$.

\section{The homology operads}  \label{sec:homology}
So far we have described operads $\NAPY$ and $\BCY$ in the ``geometric'' setting.  
Both families also have versions existing in the ``linear'' setting, so for any graded vector space $D$ there exists an operad $\NAP_D$, whereas in the case of the based cacti there is a subtlety, we require a graded augmented cocommutative coalgebra $C$ to define $\BCact_C$. The ``geometric'' and ``linear'' versions are closely related via the homology functor which sends a topological space to its homology groups with coefficients in the base field $\k$. 
In this section, we shall describe these operads via constructions with decorated rooted trees, and later in section~\ref{sec:Koszul}, we shall descibe them via generators and relations, and show that it in fact each of them has a quadratic Gr\"obner basis of relations.

\subsection{The linear operad $\NAP_D$}
Let $D$ be a graded vector space (over some field $\k$).  Recall that in~\eqref{eq_NAPYpoly} we described $\NAPY(n)$ as disjoint union of direct products of copies of $Y$.
Then in Remark~\ref{rem_polynomial} we gave a polynomial diagram~\eqref{eq_NAPpolynomial} realising $\NAPY(n)$ as a polynomial functor in $Y$. 
Let $D$ be a graded vector space and define $\NAP_D$ via the same polynomial diagram in the category of graded vector spaces:
\begin{equation}\label{eq_NAPDpoly}
\NAP_D(n) = \bigoplus_{T\in\RT(n)} D^{\otimes (n-1)}.
\end{equation}
Equivalently $\NAP_D(n)$ is the vector space spanned by rooted trees with vertex set $\sbrac{n}$ and edge labels in $D$, subject to linearity in each edge label. 

The set based description of the $\NAPY$ operad works on the level of polynomial functors and so suffices to show that $\NAP_D$ is an operad.
However great care must be taken to keep track of the signs induced by the symmetry $\sigma$ from the symmetric monoidal category $(\gVect,\otimes,\sigma,\k)$ of graded vector spaces.
In order to do this we must assign for each term $D^{\otimes n-1}$ in the sum~\eqref{eq_NAPDpoly} a reference ordering of the factors.
This requires assigning to each tree $T\in\RT(n)$ a total ordering on the set of edges $E(T)$. 
Let $T$ be such a tree and let $i$ be its root.
Since each vertex has a unique incoming edge except for the root which has none, the set of edges $E(T)$ is in bijection with the set of non-root vectices $\sbrac{n}-i$.
We take the ordering of $E(T)$ from the natural ordering of $\sbrac{n}-i$.
So for instance the pair
\begin{equation}\label{eq_Tlabelling}
\brac{\smtree{2 && 3\\ &1\ar[ur]\ar[ul] & \\ &4 \ar[u]&}, x\otimes y\otimes z}\quad\text{represents the $Y$-tree}\quad\smtree{2 && 3\\ &1\ar[ur]_z\ar[ul]^y & \\ &4 \ar[u]^x&}.
\end{equation}
The order of $x,y$ and $z$ in the tensor product is determined by the order of the edges.

The first step in giving the operad structure is to describe the action of the symmetric group $\symn$ on $\NAP_D(n)$.
For instance applying the permutation $(24)$ to the $Y$-tree considered in~\eqref{eq_Tlabelling} we get
\begin{multline}
(24). \brac{\smtree{2 && 3\\ &1\ar[ur]\ar[ul] & \\ &4 \ar[u]&}, x\otimes y\otimes z} =\\= 
\brac{\smtree{3 && 4\\ &1\ar[ur]\ar[ul] & \\ &2 \ar[u]&}, (23)x\otimes y\otimes z} 
=(-1)^{|y||z|}\smtree{3 && 4\\ &1\ar[ur]_y\ar[ul]^z & \\ &2 \ar[u]^x&}.
\end{multline}

The signs involved in the composition $T\circ_i T'$ for $T\in\NAP_D(n)$ and $T'\in\NAP_D(m)$ are more easily accounted for.  
This is because the edges within the righthand tree $T'$ are not reordered within $T\circ_i T'$ and so the sign depends on the total degree $|T'|$ and not on the individual edges.
The edges of $T'$ are `moved past' the edges $\ed{jk}\in E(T)$ for which $k>i$.  Hence if $y_{jk}$ is the labelling of $\ed{jk}$ the sign change is given by
\begin{equation}
(-1)^{\displaystyle |T'|\Bigl(\sum_{\ed{jk}\in E(T)\mid k > i} |y_{jk}|\Bigr)}.
\end{equation}

\begin{proposition}
The homology operad $H_\ast(\NAPY)$ with coefficients in the base field $k$ is isomorphic to the linear operad $\NAP_{H_\ast(Y)}$.
\end{proposition}
\begin{proof}
With field coefficients the homology functor $H_\ast$ from topological spaces to graded vector spaces respects products and coproducts and so is compatible with polynomial functors.
The explicit expression of this is
\begin{equation}
H_\ast\brac{\NAPY(n)} \cong H_\ast\Bigl(\coprod_{T\in\RT(n)} Y^{E(T)}\Bigr) \cong \bigoplus_{T\in\RT(n)} H_\ast(Y)^{\otimes E(T)}.
\end{equation}
\end{proof}

\subsection{The linear operads of based cacti}
Let $C$ be an augmented cocommutative coalgebra and write its splitting as $\k\one\oplus\overline{C}$.
The operad $\BCact_C$ will be a quotient of the operad $\NAP_C$, this is a parallel of the set-based versions.
Let $T\in\NAP_C$ be a $C$-labelled rooted tree and suppose that it has an edge $\ed{ij}$ with the label $\one$ and suppose further that $i$ is not the root of $T$, as before we will call the edge $\ed{ij}$ reducible.
Let $k$ be the unique vertex such that $\ed{ki}$ is an edge and let $c$ be the label of $\ed{ki}$.
We define $T'$ to be the unlabelled rooted tree created by removing the edge $\ed{ij}$ and replacing it by $\ed{kj}$ and denote by $T'(a,b)$ the edge labelled rooted tree based on $T'$ where the edge labels are inherited from those of $T$ except for $\ed{ki}$ which is labelled by $a$ and $\ed{kj}$ which is labelled by $b$. 
Finally we define $T_{ij}$ to be the sum
\begin{equation}
\sum (-1)^{|c_{(2)}|g}\,T'(c_{(1)},c_{(2)}),
\end{equation}
where $g$ is the sum of degrees
\begin{equation}
\sum_{\ed{xy}\mid i<y<j}|a_{xy}|
\end{equation}
and $a_{xy}$ is the label of the edge $\ed{xy}$.
The sign is given by the moving of the label $c_{(2)}$ from being adjacent to $c_{(1)}$ as in~$\Delta(c)=\sum c_{(1)}\otimes c_{(2)}$ to being in the relevant position to label the edge $\ed{kj}$.
As before $T_{ij}$ is called the reduction of $T$ at the reducible edge $\ed{ij}$ and just as before each $C$-tree reduces to a unique irreducible $C$-tree.
\begin{definition}
The graded vector space of linear based $C$-cacti, $\BCact_C$ is defined by factoring out from $\NAP_C$ the relations
\begin{equation}
T- T_{ij} = 0
\end{equation}
for trees $T$ with an edge $\ed{ij}$ labelled by $\one$ where $i$ is not the root.
\end{definition}

The graded vector space of irreducible $C$-trees and hence $\BCact_C$ is given by
\begin{equation}\label{eq_irredCtrees}    
\BCact_C(n) \cong \bigoplus_{T\in\RT(n)}  
\Bigl(\bigotimes_{\ed{r(T)j}\in E(T)} C \Bigr) \otimes  
\Bigl(\bigotimes_{{\ed{ij}\in E(T),}\\{i\neq r(T)}} \overline{C}\Bigr),  
\end{equation}
where $r(T)$ is the root of $T$.
Using the splitting $C=\k\one\oplus\overline{C}$ we may rewrite this as a polynomial expression in $\overline{C}$.
There is a convenient way of indexing this polynomial; rather than using irreducible $C$-trees, where an outgoing edge $\ed{rj}$ from the root $r$ could be labelled by $\one$, we cut the edges $\ed{rj}$ labelled by $\one$ to leave a labelled forest, each component tree has a root, the corresponding $j$ and there is a chosen component tree, the tree containing $r$.
Let $\PFs$ be the set of planted forests with a chosen tree.  Then we may rewrite~\eqref{eq_irredCtrees} as
\begin{equation}\label{eq_polyBCactC}
\BCact_C(n)\cong \bigoplus_{F\in\PFs(n)} \overline{C}^{\otimes E(F)}.
\end{equation}
\begin{remark}
Although this is a polynomial functor in $\overline{C}$ with a similar diagram to~\eqref{eq_NAPpolynomial}, the operad maps are not maps of polynomials, indeed the diagonal map of $C$ is used.  
A similar polynomial description of $\BCY$ holds when $Y$ is a set, however this involves `splitting' the chosen point of $Y$ and so this only works for a pointed topological space when the point is disconnected.
\end{remark}

\begin{proposition}\label{prop:cacti-distr}
The linear subspace of $\NAP_C$ generated by relations of the form $T-T_{ij}=0$ is an operadic ideal and so $\BCact_C$ is an operad as a quotient of $\NAP_C$.
Furthermore the ideal is generated in arity 3 by
\begin{equation}
\smtree{3 \\ 2\ar[u]_\one\\ 1\ar[u]_c} \quad-\quad \sum\smtree{2 && 3 \\ &1\ar[ul]^{c_{(1)}} \ar[ur]_{c_{(2)}}& }\quad =\quad 0.
\end{equation}
\end{proposition}
\begin{proof}
This is the linear analogue of Theorem~\ref{thm:BCYoperad} and the same method applies.
\end{proof}

In the linear setting, the formula~\eqref{eq:NAP-via-family} (and its particular case~\eqref{eq:NAP-via-hom}), as well as~\eqref{eq:free-NAP} work without any changes (except for signs that one should carefully trace), while the formula~\eqref{eq:BCact-via-family} should be adapted into
\begin{equation}
f(c,\smtree{2\\1\ar[u]_\one}(a,b))=\sum f(c_{(1)},a)\circ f(c_{(2)},b). 
\end{equation}

\begin{theorem}\label{thm:homology-BCY}
Let $Y$ be a topological space then the homology operad of $\BCY$ is isomorphic to the linear operad $\BCact_{H_\ast(Y)}$.
\end{theorem}
\begin{proof}
The homology functor respects products and coproducts and hence polynomial functors, this is how we see that $H_\ast\brac{\NAP_Y} \cong \NAP_{H_\ast(Y)}$.
However the cactus operad $\BCY$ is not given by a polynomial functor.
\comment{  
We defined the cactus operads~--- $\BCY$ for a pointed space $Y$ and $\BCact_C$ for a graded augmented cocommutative coalgebra $C$~--- by adding certain relations $T\sim T_{ij}$ for trees $T$ with a reducible edge $\ed{ij}$.
The relations come in families: for a fixed tree $T$ with a fixed edge $\ed{ij}$ where $i$ is not the root, a $Y$-tree is reducible if $\ed{ij}$ is labelled by the point $\dt\in Y$ and the remaining $n-2$ edges are labelled by any element in $Y$, so there is a family of relations parametrised by $\cbrac{\dt}\times Y^{n-2}$.
Each element in this family encodes a reduction $T\sim T_{ij}$:  there is one map from $Y^{n-2}$ corresponding to $T$ and another map from $Y^{n-2}$ corresponding to $T_{ij}$.  
For the second map the diagonal map $y\mapsto (y,y)$ is used to define the new labelling.
So for each such tree $T$ with edge $\ed{ij}$ there are a pair of maps
\begin{equation}\xymatrix{
Y^{n-2}\ar@<0.5ex>[r]\ar@<-.5ex>[r] & \NAP_Y(n).}
\end{equation}
In identifying the two images of each point we are taking the coequaliser of this diagram.
But we have such an identification for each tree $T$ with an edge $\ed{ij}$ where $i$ is not the root.
So we have a diagram with a copy of $Y^{n-2}$ for each such pair and two arrows from each copy of $Y^{n-2}$ to a single copy of $\NAP_Y(n)$.
The colimit of this diagram is the space given by making all identifications $T\sim T_{ij}$ -- that is, the colimit is $\BCY(n)$.
}  
In the proof of~Proposition~\ref{prop:BCYgroupoid} we showed that $\BCY(n)$ was the colimit of a diagram containing a single copy of $\NAP_Y(n)$ and a copy of $Y^{n-2}$ for each pair $(T,\ed{ij})$ where $T\in\RT(n)$ and $\ed{ij}\in E(T)$ where $i$ is not the root.
For each copy of $Y^{n-2}$ there were two maps, one corresponding to $T$ and one to its reduction $T_{ij}$.
The act of taking the colimit makes identifications $T\sim T_{ij}$.

Precisely the same discussion applies to the linear operad $\BCact_C(n)$; realising it as the colimit of the same diagram but with $C^{\otimes n-2}$ replacing $Y^{n-2}$ and $\NAP_C(n)$ replacing $\NAP_Y(n)$.

Unfortunately the homology functor $H_\ast$ preserves coproducts and products, but not general colimits.  Therefore we can not just apply the homology functor to the colimit diagram for $\BCY(n)$.
The chain functor $C_\ast$ which takes values in the symmetric monoidal category of differentially graded vector spaces does however preserve colimits.
Therefore $C_\ast(\BCY(n))$ is the colimit of the diagram consisting of $C_\ast(\NAP_Y(n))$ and copies of $C_\ast(Y^{n-2})$.
However $C_\ast$ does not preserve products which is inconvenient because $C_\ast(Y)$ can not be assumed to be a coalgebra, although we still have the diagonal maps $C_\ast(Y)\rightarrow C_\ast(Y\times Y)$ which allow reductions to be made.
Since $Y$ is pointed there is a natural splitting $C_\ast(Y) \cong \k\one\oplus\overline{C}_\ast(Y)$ and furthermore the inclusions $Y^{a}\rightarrow Y^{b}$ induce splittings $C_\ast(Y^b) \cong C_\ast(Y^a)\oplus \overline{D}$.
The most general splitting is given by taking the kernel of the map 
\begin{equation}
C_\ast(Y^b) \rightarrow \bigoplus_{i=1\clc b} C_\ast(Y^{b-1}),
\end{equation}
where the $i$th map forgets the $i$th coordinate, call this kernel $\overline{C_\ast(Y^b)}$.  Then $C_\ast(Y^b)$ undergoes the splitting:
\begin{equation}
C_\ast(Y^b) \cong \bigoplus_{A\subseteq\sbrac{b}} \overline{C_\ast(Y^{|A|})}.
\end{equation}
This splitting allows one to compute the colimit of the diagram computing the space~$C_\ast(\BCY(n))$ in the same manner as the computation of $\BCact_C(n)$ in~\eqref{eq_polyBCactC}, using reductions $T\sim T_{ij}$ as before.  
Therefore 
\begin{equation}
C_\ast\brac{\BCact_Y(n)} \cong \bigoplus_{F\in\PFs(n)} \overline{C_\ast\brac{Y^{\otimes |E(F)|}}}.
\end{equation}
The homology functor does not necessarily preserve colimits but it does preserve products and hence
\begin{equation}
H_\ast\brac{ \overline{ C_\ast\brac{Y^{b}} } } \cong \overline{H}_\ast(Y)^{\otimes b}.
\end{equation}
Therefore
\begin{equation}
H_\ast\brac{\BCact_Y(n)} \cong \bigoplus_{F\in\PFs(n)} \overline{H}_\ast(Y)^{\otimes |E(F)|}.
\end{equation}
Which is isomorphic to $\BCact_{H_\ast(Y)}(n)$.
That this is an isomorphism of operads is immediate because both cacti operads are defined as quotients of NAP operads.
\end{proof}

\begin{remark}
Since when $C=H_*(Y)$, the operad $\BCact_C$ is the homology operad of a topological operad, it should not be surprising at all that for every coalgebra~$C$ the operad $\BCact_C$ is a Hopf operad~\cite{GJ,Moe99}, which essentially means that algebras over it form a tensor category. Its diagonal map coincides with the diagonal of the coalgebra~$C$ on the space of generators:
\begin{equation}
\Delta(\smtree{2\\1\ar[u]_c})=\sum \smtree{2\\1\ar[u]_{c_{(1)}}}\otimes \smtree{2\\1\ar[u]_{c_{(2)}}}. 
\end{equation}
\end{remark}

Let us conclude this section with an example of a ``smallest nontrivial algebra'' over a linear operad of based cacti.

\begin{example}\label{ex:perm-nap}
Let $Y$ be the (pointed) two-element set~$\{\mathbf{0},1\}$, so that $C=H_*(Y)$ is the split two-dimensional coalgebra $\k\oplus\k$, the product $\cdot_0$ defines a $\Perm$-algebra, and the product $\cdot_1$ defines an $\NAP$-algebra. In every one-dimensional $\BCact_C$-algebra, the $\Perm$-product is commutative, and the $\NAP$-product is associative, so they are very degenerate, and the first nontrivial example should be at least two-dimensional. 
One can easily check that a two-dimensional \emph{noncommutative} $\Perm$-algebra is necessarily isomorphic to the algebra $A=\cbrac{a,b}$ with multiplication table
\begin{gather}
a\cdot_0 a=a,\\
a\cdot_0 b=b\cdot_0 b=0,\\
b\cdot_0 a=b.
\end{gather}
Furthermore, to define a $\BCact_C$-algebra structure on~$A$, one should choose a $2\times2$-matrix $p$ with $p^2=p$, and put
\begin{gather}
a\cdot_1 a=p_{11}a+p_{12}b,\\
b\cdot_1 a=p_{21}a+p_{22}b,\\
a\cdot_1 b=b\cdot_1 b=0.
\end{gather}
One particular example will be obtained if we put $p=\begin{pmatrix}0&0\\0&1\end{pmatrix}$, so that the $\NAP$-product in this algebra is given by
\begin{gather}
a\cdot_1 a=a\cdot_1 b=b\cdot_1 b=0,\\
b\cdot_1 a=b.
\end{gather}
This product is ``nontrivial'' enough: it has a noncommutative $\Perm$-product, a nonassociative $\NAP$-product, and moreover it does not fit into the series of algebras defined in Example~\ref{ex:NAP-from-Perm} (since we have $a\cdot_1 a=0$ but $b\cdot_1 a=b\ne0$).
\end{example}

\section{Filtered distributive laws}\label{sec:distr-laws}

\subsection{Filtered distributive laws between quadratic operads}

Assume that $\calA=\calF(\calV)/(\calR)$ and $\calB=\calF(\calW)/(\calS)$ are two quadratic operads. For two subspaces $\calU_1$ and $\calU_2$ of the same operad $\calO$, let us denote by $\calU_1\bullet\calU_2$ the subspace of $\calO$ spanned by all elements $\phi\circ_i\psi$ with $\phi\in\calU_1$, $\psi\in\calU_2$. For two $\mathbb{S}$-module mappings
\begin{equation}
s\colon \calR_{(2)}\rightarrow\calW\bullet\calV\oplus\calV\bullet\calW\oplus\calW\bullet\calW 
\end{equation}
and
\begin{equation}
d\colon\calW\bullet\calV\rightarrow\calV\bullet\calW\oplus\calW\bullet\calW,
\end{equation}
one can define a quadratic operad $\calE$ with generators $\calU=\calV\oplus\calW$ and relations
$\calT=\calQ\oplus\calD\oplus\calS$, where 
\begin{equation}
\calQ=\{x-s(x)\mid x\in\calR_{(2)}\},\quad \calD=\{x-d(x)\mid x\in\calW\bullet\calV\}.
\end{equation}

Informally, we join generators of $\calA$ and $\calB$ together, keep the relations of~$\calB$, deform relations of $\calA$, adding to them ``lower terms'' of degree at most $1$ in generators of $\calA$, and impose a rewriting rule transforming $\calW\bullet\calV$ into a combination of terms from $\calV\bullet\calW$ and ``lower terms'' of degree~$0$ in generators of $\calA$. Note that using the rewriting rule~$x\mapsto d(x)$, one can replace~$s$ by 
\begin{equation}
s'\colon \calR_{(2)}\rightarrow\calV\bullet\calW\oplus\calW\bullet\calW,
\end{equation}
and from now on we shall denote by~$s$ that modified mapping.

Assume that the natural projection of $\mathbb{S}$-modules $\pi\colon\calE\twoheadrightarrow\calA$ splits (for example, it is always true in characteristic zero, or in arbitrary characteristic whenever the relations of $\calA$ remain undeformed, including the case of usual distributive laws). Then
the composite of natural mappings
\begin{equation}
\calF(\calV)\circ\calF(\calW)\hookrightarrow\calF(\calV\oplus\calW)\twoheadrightarrow\calF(\calV\oplus\calW)/(\calT)
\end{equation}
gives rise to a surjection of $\mathbb{S}$-modules
\begin{equation}
\xi\colon\calA\circ\calB\twoheadrightarrow\calE.
\end{equation}

\begin{definition}\label{def:filtered-distrib}
We say that the mappings $s$ and $d$ above define a filtered distributive law between the operads $\calA$ and $\calB$ if $\pi\colon\calE\twoheadrightarrow\calA$ splits, and the restriction of $\xi$ to weight~$3$ elements
\begin{equation}
\xi_3\colon(\calA\circ\calB)_{(3)}\to\calE_{(3)}
\end{equation}
is an isomorphism.
\end{definition}

The following result (generalising the distributive law criterion for operads that was first stated in~\cite{Markl94}) was proved in~\cite{Dotsenko2006} using the set operad filtration method of~\cite{Khoroshkin2006} and in \cite{Vallette_PROP} using a filtration on the Koszul complex; however, both proofs rely on the K\"unneth formula for symmetric collections and thus are not available in positive characteristic because in that case the group algebras $\k\mathbb{S}_n$ are not semisimple. 

\begin{theorem}\label{thm:filtered-distrib}
Assume that the operads $\calA$ and~$\calB$ are Koszul, and that the mappings $s$ and $d$ define a filtered distributive law between them. Then the operad $\calE$ is Koszul, and the $\mathbb{S}$-modules $\calA\circ\calB$ and $\calE$ are isomorphic.
\end{theorem}

\begin{proof}
Let us first note that either of the characteristic zero proofs mentioned above (set operad filtration; filtration on the Koszul complex) works in the category of shuffle operads for arbitrary characteristic, since K\"unneth formula over a field is always available. Also, a symmetric operad $\calO$ is Koszul if and only if it is Koszul as a shuffle operad, which proves the first statement of the theorem. To prove the second statement, we observe that in the category of nonsymmetric collections we have an isomorphism $\calE^f\simeq\calA^f\circ_{sh}\calB^f\simeq(\calA\circ\calB)^f$, and in the symmetric category we have a surjection $\calA\circ\calB\twoheadrightarrow\calE$. Since the forgetful functor from the category of symmetric collections to the category of nonsymmetric collections is one-to-one on objects, that surjection has to be an isomorphism.
\end{proof}

\begin{example}
The following filtered distributive law was discussed by the first author in \cite{Dotsenko2006} as related to Gelfand--Varchenko algebras of locally constant functions on the complement to a hyperplane arrangement; unlike all other results of this paper, it is only available in characteristic zero. It is well known (and was probably first observed by Livernet and Loday) that the associative operad admits an alternative description as an operad generated by a symmetric binary operation $\cdot\star\cdot$ and a skew-symmetric binary operation $[\cdot,\cdot]$ that satisfy the relations
\begin{gather}
[a,[b,c]]+[b,[c,a]]+[c,[a,b]]=0,\\
[a\star b,c]=a\star[b,c]+[a,c]\star b,\\
(a\star b)\star c-a\star(b\star c)=[b,[a,c]]. 
\end{gather}
If we put $\calV=\SPAN(\cdot\star\cdot)$, $\calW=\SPAN([\cdot,\cdot])$, and consider the operads $\calA=\Com$ and $\calB=\Lie$, 
\begin{gather}
s((a\star b)\star c-a\star(b\star c))=[b,[a,c]],\\ 
d([a\star b,c])=a\star[b,c]+[a,c]\star b
\end{gather}
then there are no additional relations in weight~$3$, and in characteristic zero the projection $\As\twoheadrightarrow\Com$ splits, therefore the associative operad is built from $\Com$ and $\Lie$ via a filtered distributive law. Thus we obtain a yet another proof of the Koszulness of the associative operad, and also recover that as an $\mathbb{S}$-module it is isomorphic to $\Com\circ\Lie$.
\end{example}

\subsection{Filtered distributive laws and Koszul duality}

An easy linear algebra exercise shows that if $\calE$ is obtained from $\calA$ and $\calB$ via the mappings $s$ and $d$ as above, then the Koszul dual operad~$\calE^!$ is similarly obtained from $\calB^!$ and~$\calA^!$. The following result shows that the notion of a filtered distributive law agrees very well with the Koszul duality theory for operads (which our previous example --- being Koszul self-dual --- did not quite manifest). 

\begin{theorem}\label{thm:koszul-dual}
Assume that the operad $\calE$ is obtained from the binary quadratic operads $\calA$ and $\calB$ via a filtered distributive law. Then its Koszul dual~$\calE^!$ is obtained from $\calB^!$ and $\calA^!$ by a filtered distributive law as well whenever the projection $\calE^!\twoheadrightarrow\calB^!$ splits. 
\end{theorem}

\begin{proof}
If both operads~$\calA$ and~$\calB$ are Koszul, then $\calE$ is Koszul, and this gives us enough information to complete the proof, see~\cite{Dotsenko2006} for details. Let us give a proof in the case of arbitrary~$\calA$ and~$\calB$ to show a yet another application of methods developed in~\cite{DKRes}.

Let us define an ordering on tree monomials in the free shuffle operad generated by~$\calV^f\oplus\calW^f$ in the following way. For two tree monomials, we first compute the number of generators from $\calV^f$ used in each of them; if for one of them that number is greater than for the other, we say that monomial is greater than the other. Otherwise, we compare tree monomials using the lexicographic ordering on paths~\cite{DK,DVJ}. This way we can be sure that the leading monomials of~$\calR^f$, tree monomials spanning $\calW^f\bullet\calV^f$, and the leading monomials of~$\calS^f$ are the leading monomials of the defining relations of~$\calE$. 

Since the $\mathbb{S}$-module $\calE$ is a quotient of $\calA\circ\calB$, so the distributive law condition ensures that the set of weight~$3$ leading monomials of the reduced Gr\"obner basis of~$\calE^f$ is the union of the set of weight~$3$ leading monomials of the reduced Gr\"obner basis of~$\calA^f$ and the set of weight~$3$ leading monomials of the reduced Gr\"obner basis of~$\calB^f$: the presence of ``mixed'' leading monomials would make $\calE_{(3)}$ smaller than its natural upper bound $(\calA\circ\calB)_{(3)}$. In other words, all S-polynomials \cite{DK} of weight~$3$ of $\calE^f$ are either S-polynomials of~$\calA^f$ or S-polynomials of~$\calB^f$.

The above description of leading monomials of the reduced Gr\"obner basis means that we have the full information on the part of the free resolution of~$\calE^f$ consisting of elements of weight at most~$3$, and a simple description of the homology classes of the bar complex of~$\calE^f$ up to weight~$3$. From \cite{DKRes}, we know that generators of a free resolution of~$\calE^f$ can be constructed in terms of ``overlaps'' of leading monomials of the reduced Gr\"obner basis of~$\calE^f$. Such generators of weight~$2$ are precisely the leading monomials of the defining relations, whereas the generators of weight~$3$ are either overlaps of pairs of leading monomials of defining relations or leading monomials of weight~$3$ elements of the reduced Gr\"obner basis. The differential induced on the space of the generators of that free resolution can be computed as follows. If an overlap of two leading monomials of defining relations produces, according to Buchberger's algorithm~\cite{DK}, a nontrivial S-polynomial, the differential maps the generator corresponding to that overlap to the generator corresponding to the leading term of the respective S-polynomial. Otherwise, the differential maps the corresponding generator to zero. Together with the information on S-polynomials of~$\calE^f$ that we have, this means that up to weight~$3$ the homology of the bar complex of~$\calE^f$ is isomorphic to the shuffle composition of the corresponding homology for $\calB^f$ and $\calA^f$. Since the Koszul dual operads are dual to the diagonal parts of the bar homology, our statement follows in the shuffle category. In the symmetric category, we observe that because of the splitting of $\calE^!\twoheadrightarrow\calB^!$, there is a surjection
$\calB^!\circ\calA^!\twoheadrightarrow\calE^!$, and its bijectivity in weight~$3$ in the shuffle category implies bijectivity in the symmetric category as well.
\end{proof}

\subsection{Operadic K\"unneth formula}

We conclude this section with a general observation which appears to be useful for transferring statements of the characteristic zero operad theory in positive characteristic. If one examines the proof of Theorem~\ref{thm:filtered-distrib} carefully, it becomes obvious that it works because of the following statement, a particular case of the operadic K\"unneth formula~\cite{LodayVallette2011}, which is valid over any ground field~$\k$.

\begin{theorem}
Let $\calM$ and $\calN$ be two reduced differential graded $\mathbb{S}$-modules. Then 
\begin{equation}\label{kunneth}
H_*(\calM\circ\calN)\simeq H_*(\calM)\circ H_*(\calN).
\end{equation}
\end{theorem}

\begin{proof}
Let us note that there is a natural map
\begin{equation}
\kappa\colon H_*(\calM)\circ H_*(\calN)\to H_*(\calM\circ\calN).
\end{equation}
Our strategy is to apply the forgetful functor, and prove that 
\begin{equation}
\kappa^f\colon (H_*(\calM)\circ H_*(\calN))^f\to (H_*(\calM\circ\calN))^f
\end{equation}
is an isomorphism in the shuffle category. Since the forgetful functor is one-to-one on objects, this would mean that $\kappa$ is an isomorphism.
In the shuffle category, since the forgetful functor is monoidal (that is the only part of the proof where it is crucial that our collections are reduced), we have
\begin{equation}
(H_*(\calM)\circ H_*(\calN))^f\simeq (H_*(\calM))^f\circ_{sh} (H_*(\calN))^f\simeq H_*(\calM^f)\circ_{sh} H_*(\calN^f),
\end{equation}
and
\begin{equation}
H_*(\calM\circ\calN)^f\simeq H_*((\calM\circ\calN)^f)\simeq H_*(\calM^f\circ_{sh}\calN^f).
\end{equation}
Finally, since the shuffle composition is polynomial in the components of $\calM^f$ and $\calN^f$, 
we have
\begin{equation}
H_*(\calM^f\circ_{sh}\calN^f)\simeq H_*(\calM^f)\circ_{sh} H_*(\calN^f),
\end{equation}
and the theorem follows.
\end{proof}

\section{Koszulness of cacti and other operads}\label{sec:Koszul}

In this section, we prove that the operads $\NAP_D$ and $\BCact_C$ are Koszul, and also show how one can use filtered distributive laws to recover known results, and obtain new results on the structure of various known operads.

\subsection{The operad $\PostLie$}

The operad $\PostLie$ was defined and studied in \cite{ChapotonVallette2006,Vallette2007}, and recently appeared in various contexts, see \cite{PostLie0,PostLie1,PostLie2,PostLie3}. It is generated by a skew-symmetric operation $[\cdot,\cdot]$ and an operation $\cdot\circ\cdot$ without any symmetries that satisfy the relations
\begin{gather}
[a,[b,c]]+[b,[c,a]]+[c,[a,b]]=0,\\
(a\circ b)\circ c-a\circ(b\circ c)-(a\circ c)\circ b+a\circ(c\circ b)=a\circ[b,c],\\
[a,b]\circ c=[a\circ c,b]+[a,b\circ c]. 
\end{gather}

The Koszul dual $\PostLie^!=\ComTrias$ by commutative trialgebras is generated by a symmetric operation $\cdot\bullet\cdot$ and an operation $\cdot\star\cdot$ without any symmetries that satisfy the relations
\begin{gather}
(a\star b)\star c=a\star (b\star c)=a\star (c\star b),\\
(a\bullet b)\bullet c=a\bullet(b\bullet c),\\
a\star(b\star c)=a\star(b\bullet c),\\
a\bullet(b\star c)=(a\bullet b)\star c.
\end{gather}

\begin{theorem}
The operad $\PostLie$ is Koszul, and as an $\mathbb{S}$-module is isomorphic to $\Lie\circ\Mag$. 
\end{theorem}

\begin{proof}
By an immediate computation, we see that the operad $\PostLie$ is built from the operads $\calA=\Lie$ and $\calB=\Mag$ via a filtered distributive law. Indeed, we may put $\calV=\SPAN([\cdot,\cdot])$, $\calW=\SPAN(\cdot\circ\cdot)$, and
\begin{gather}
s([a,[b,c]]+[b,[c,a]]+[c,[a,b]])=0,\\
d([a,b]\circ c)=[a\circ c,b]+[a,b\circ c],\\
d(a\circ[b,c])=(a\circ b)\circ c-a\circ(b\circ c)-(a\circ c)\circ b+a\circ(c\circ b)
\end{gather}
(the weight~$3$ condition can be easily checked by hand, and since $s=0$, the projection is split automatically). This proves both statements of our theorem. 
\end{proof}

The Koszulness of $\PostLie$ and $\PostLie^!=\ComTrias$ was established in \cite{ChapotonVallette2006} using partition posets. Note that our approach applies to $\ComTrias$ as well, since the splitting of the projection $\ComTrias\twoheadrightarrow\Mag^!=\mathop{\mathrm{Nil}}$ only requires the splitting on the level of generators, which we already have. The $\mathbb{S}$-module isomorphism $\PostLie\simeq\Lie\circ\Mag$ was first observed in \cite{Vallette2007}\footnote{The proof in the published version of that paper is incomplete (one has to check that the extension of $\cdot\circ\cdot$ to the free algebra $\Lie(\Mag(V))$ is consistent with the Jacobi identity).}. This isomorphism, together with the following corollary, allows to complete the PostLie algebras description in~\cite{LodayEnc}.

\begin{corollary}
The suboperad of $\PostLie$ generated by $\cdot\circ\cdot$ is isomorphic to $\Mag$. 
\end{corollary}

Note that the dual version of this corollary is not true: even though on the level of $\mathbb{S}$-modules we have $\ComTrias\simeq\mathop{\mathrm{Nil}}\circ\Com$, it is easy to check the suboperad of the operad $\ComTrias$ generated by the operation $\cdot\star\cdot$ is isomorphic to $\Perm$.

\subsection{The operad of commutative tridendriform algebras} 

The operad of commutative tridendriform algebras was studied by Loday \cite{Loday2007}. Let us write down the relations of this operad, and of its Koszul dual. The operad $\CTD$ is generated by a symmetric operation $\cdot\star\cdot$ and an operation $\cdot\prec\cdot$ without any symmetries that satisfy the relations
\begin{gather}
(a\prec b)\prec c=a\prec(b\prec c+c\prec b+b\star c),\\
(a\star b)\prec c=a\star(b\prec c)=(a\prec c)\star b,\\
(a\star b)\star c=a\star(b\star c).  
\end{gather}
The operad $\CTD^!$ is generated by a skew-symmetric operation $[\cdot,\cdot]$ and an operation $\cdot\bullet\cdot$ without any symmetries that satisfy the relations
\begin{gather}
[a,[b,c]]+[b,[c,a]]+[c,[a,b]]=0,\\
a\bullet[b,c]=a\bullet(b\bullet c),\\
[a,b]\bullet c=[a\bullet c,b]+[a,b\bullet c],\\
(a\bullet b)\bullet c=a\bullet(b\bullet c)+(a\bullet c)\bullet b.
\end{gather}

\begin{theorem}
\begin{itemize}
 \item The operad $\CTD$ is Koszul, and as an $\mathbb{S}$-module is isomorphic to $\Zinb\circ\Com$. 
 \item  The operad $\CTD^!$ is Koszul, and as an $\mathbb{S}$-module is isomorphic to $\Lie\circ\Leib$. 
\end{itemize}
\end{theorem}

\begin{proof}
By an immediate computation, we notice that the operad $\CTD$ is built from the operad $\Zinb$ and $\Com$ via a filtered distributive law. Indeed, we may put $\calV=\SPAN(\cdot\prec\cdot)$, $\calW=\SPAN(\cdot\star\cdot)$, and 
\begin{gather}
s((a\prec b)\prec c-a\prec(b\prec c+c\prec b))=a\prec(b\star c),\\
d(a\star(b\prec c))=(a\star b)\prec c,\\
d((a\prec c)\star b)=(a\star b)\prec c.
\end{gather}
(the weight~$3$ condition can be easily checked by hand; the projection $\CTD\twoheadrightarrow\Zinb$ splits because $\Zinb(n)$ is a free $\mathbb{S}_n$-module).
Therefore Theorems \ref{thm:filtered-distrib} and~\ref{thm:koszul-dual} prove all the statements of our theorem (for the latter, we observe that the projection $CTD^!\twoheadrightarrow\Lie$ splits because for $\CTD^!$ we have $s=0$). 
\end{proof}

The $\mathbb{S}$-module isomorphism in the first part was proved in~\cite{Loday2007} as a consequence of the existence of a good triple of operads $(\As,\CTD,\Com)$ and the isomorphism of $\mathbb{S}$-modules $\As\simeq\Zinb$. Our results recover that isomorphism, prove a similar isomorphism for $\CTD^!$, and also describe the sub-operads of $\CTD$ and $\CTD^!$ generated by either one of the operations. This provides the following bits of information that have been missing in~\cite{LodayEnc}.

\begin{proposition}
\begin{itemize}
 \item The generating series of the operad of dual commutative tridendriform algebras is equal to
\begin{equation}
f^{\CTD^!}(t)=-\log\left(\frac{1-2t}{1-t}\right).
\end{equation}
 \item The suboperad of $\CTD^!$ generated by the operation $\cdot\bullet\cdot$ is isomorphic to $\Leib$. 
\end{itemize}
\end{proposition}

Note that though the underlying $\mathbb{S}$-module of the operad $\Zinb=\Leib^!$ is used in the definition of the operad~$\CTD$, the dual statement to the second part of this proposition is not true: in the operad $\CTD$, the suboperad generated by $\cdot\prec\cdot$ is not isomorphic to~$\Zinb$ because of the ``lower term'' $a\prec(b\star c)$ added to the Zinbiel relation.

\subsection{The linear $\NAP_D$ operad}

\begin{proposition}\label{prop:relations-NAP}
Let $D$ be a graded vector space. The operad $\NAP_D$ is generated by binary operations $D\otimes\k\mathbb{S}_2$; these operations satisfy the relations 
\begin{equation}\label{eq:NAP_graded_vector_space}
d'\circ_1 d''.(23) = (-1)^{|d'||d''|} d''\circ_1 d'\quad  (\text{for homogeneous } d',d''\in D).\\ 
\end{equation} 
\end{proposition}

\begin{proof}
The ``geometric'' version of this proposition is proved as part of Proposition~\ref{prop:NAPoperad}.  
That the linear version is generated by binary operations may be proved by precisely the same method.
As before the relations just express the symmetric group action on trees:
\begin{equation}
(23).\smtree{3 && 2 \\ &1\ar[ul]^{d'} \ar[ur]_{d''}& }=(-1)^{|d'||d''|}\smtree{2 && 3 \\ &1\ar[ul]^{d'} \ar[ur]_{d''}& }. 
\end{equation}
\end{proof}

\begin{theorem}\label{thm:NAP-Koszul}
The operad $\NAP_D$ is Koszul.
\end{theorem}

\begin{proof}
Note that according to Proposition~\ref{prop:relations-NAP}, the operad~$\NAP_D$ is a quotient of the operad $\calN_D$ generated by binary operations $D\otimes\k\mathbb{S}_2$ subject only to relations \eqref{eq:NAP_graded_vector_space}. Let us show that the operad~$\calN_D$ is Koszul, and is isomorphic to~$\NAP_D$. 

First of all, one can easily check that the Koszul dual $\calN_D^!$ of the operad~$\calN_D$ has generators $D^*\otimes\k\mathbb{S}_2$ subject to relations
\begin{gather}\label{eq:NAPdual_graded_vector_space}
e'\circ_1 e'' = (-1)^{|e'||e''|} e''\circ_1 e'.(23)\quad  (\text{for homogeneous } e',e''\in D^*),\\ 
e'\circ_2 e'' = 0. 
\end{gather} 
This immediately implies that if we choose a basis $e_1,\ldots,e_n$ of $D^*$, then for a basis of $\NAP_D^!(1)$ we can take the set of all ``left combs'' 
\begin{equation}
(e_{i_1}\circ_1 e_{i_2}\circ_1\cdots\circ_1 e_{i_{n-1}}).(1,k,k-1,\ldots,2),
\end{equation}
because our relations mean that the tree monomials can only ``grow'' to the left, and that we can reorder all elements except for the leftmost one arbitrarily. There are $(\dim D)^{n-1}\cdot n$ such monomials. At the same time, if we explicitly write the relations of $\calN_D$ as a shuffle operad, we see that its relations are
\begin{gather}
d'\circ_1 d''.(23)=(-1)^{|d'||d''|}d''\circ_1 d',\\ 
d'\circ_1 \tilde{d''}.(23)=(-1)^{|d'||d''|}\tilde{d''}\circ_2\tilde{d'},\\
d'\circ_1 \tilde{d''}=(-1)^{|d'||d''|}\tilde{d''}_j\circ_2 d'.
\end{gather}
Here we use the notation $\tilde{d}$ to abbreviate the ``opposite operation'' $d\otimes\sigma\in D\otimes\k\mathbb{S}_2$. 

Let us pick a basis $d_1,\ldots,d_n$ of~$D$, and define an ordering of tree monomials in the free shuffle operad with binary generators $D\otimes\k\mathbb{S}_2$ which is very similar to the path-lexicographic ordering~\cite{DK}. For two tree monomials, we first compare lexicographically their sequences of leaves, read left-to-right, and then compare the path sequences of those monomials, assuming 
\begin{equation}
d_1<\ldots<d_n<d_1.(12)<\ldots<d_n.(12).
\end{equation}

The leading monomials of the relations of $\calN_D$ are, respectively, $d_i\circ_1 d_j.(23)$, $d_i\circ_1 \tilde{d}_j.(23)$, and $d_i\circ_1 \tilde{d}_j$. The trees built from these monomials as building blocks give an upper bound on the dimensions of components of the Koszul dual operad which is sharp precisely when our operads have quadratic Gr\"obner bases \cite{DFree}. It is easy to see that there are exactly $(\dim D)^{n-1}\cdot n$ tree monomials built from these, so both the operads~$\calN_D$ and $\calN_D^!$ are Koszul. Power series inversion equation for Koszul operads~\cite{GK} implies that 
\begin{equation}
f_{\calN_D^!}(-f_{\calN_D}(-t))=t.  
\end{equation}
Since it is clear that
\begin{equation}
f_{\calN^!_D}(t)=\sum_{n\ge1}\frac{(\dim D)^{n-1}}{n!}t^n, 
\end{equation}
after denoting $g(s):=\frac{f_{\calN_D}(\dim D\cdot s)}{\dim D}$, we see that $g(-s)$ is the inverse of $-s\exp(-s)$ under composition, and hence $g(s)$ is the generating function enumerating rooted trees. Recalling that $\NAP_D$ as an~$\mathbb{S}$-module is described as $D$-decorated rooted trees, we conclude that components of $\calN_D$ and $\NAP_D$ have same dimensions, and therefore these operads are isomorphic, the former being a quotient of the latter.
\end{proof}
This proof concluded by showing that $\NAP_D$ is presented by quadratic relations.  
By considering the linearization of the operad $\NAP_F$ when $F$ is a finite set we see that $\NAP_F$ is also presented by quadratic relations.
Now suppose that $Y$ is infinite.  Any finite set of $Y$-trees involves a finite number of labels $F$ and hence any relation in $\NAP_Y$ is contained within $\NAP_F$ which is in turn presented by quadratic relations.  Therefore we have the following.
\begin{corollary}\label{cor:NAPYquadratic}
Let $Y$ be a topological space.  Then the operad $\NAP_Y$ is generated by binary operations and is presented by its quadratic relations.
\end{corollary}

\begin{remark}
The proof of Theorem~\ref{thm:NAP-Koszul} used arguments involving the Koszul dual and its Hilbert series to show that the quadratic relations suffices to present $\NAP_Y$.
A more direct proof is possible using a certain ``geometric'' map from $\mathcal{F}(\NAP_Y(2))$ to $\NAP_Y$.
We will denote elements of $\NAP_Y(2)$ by
\begin{equation}
\vcenter{\hbox{\includegraphics[height=12mm]{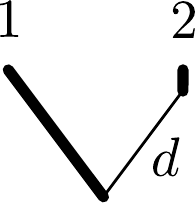}}}
\quad\text{ and }\quad
\vcenter{\hbox{\includegraphics[height=12mm]{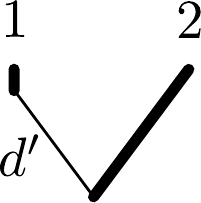}}}
\end{equation}
for $d,d'\in Y$.  In each generator there is a thin line labelled with an element of $Y$, a thick line running from root to the end of a leaf and a small portion of thick line at the end of the other leaf. Then the $\NAP_Y$-relation~\eqref{eq_NAPrel} states that
\begin{equation}
d'\circ_1 d \quad = \quad
\vcenter{\hbox{\includegraphics[height=20mm]{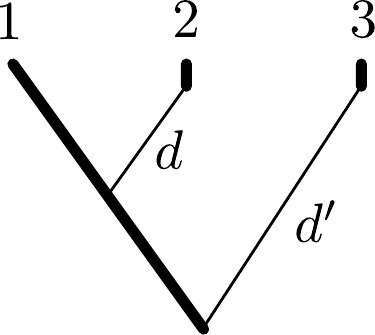}}}
\quad = \quad
\vcenter{\hbox{\includegraphics[height=20mm]{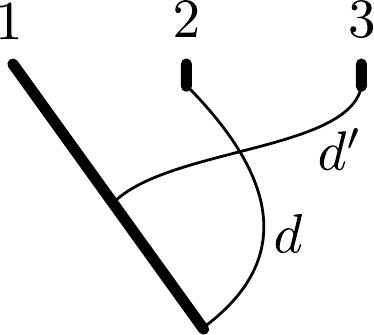}}}
\quad = \quad d\circ_1 d' . (23).
\end{equation}
The thin lines may be seen to ``move freely'' along the thick lines.
A couple of facts are apparent about any arity $n$ tree monomial in these generators:
\begin{enumerate}
\item The thick lines never branch and each thick line can be followed up the tree to a unique leaf, in this way the thick lines are in bijection with the leaves.
\item Every thin line joins two thick lines and is labelled by an element of $Y$.
\end{enumerate}
So by contracting each thick line to a point and using these as vertices we are left with a tree with vertex set $\sbrac{n}$.  The thin lines become the edges and are already labelled by elements of $Y$.  This tree is rooted by following the thick line starting at the bottom of the tree monomial to its leaf.  Hence we have an explicit map from $\mathcal{F}(\NAP_Y(2))$ to $\NAP_Y$. An example:
\begin{equation}
\vcenter{\hbox{
\includegraphics[height=35mm]{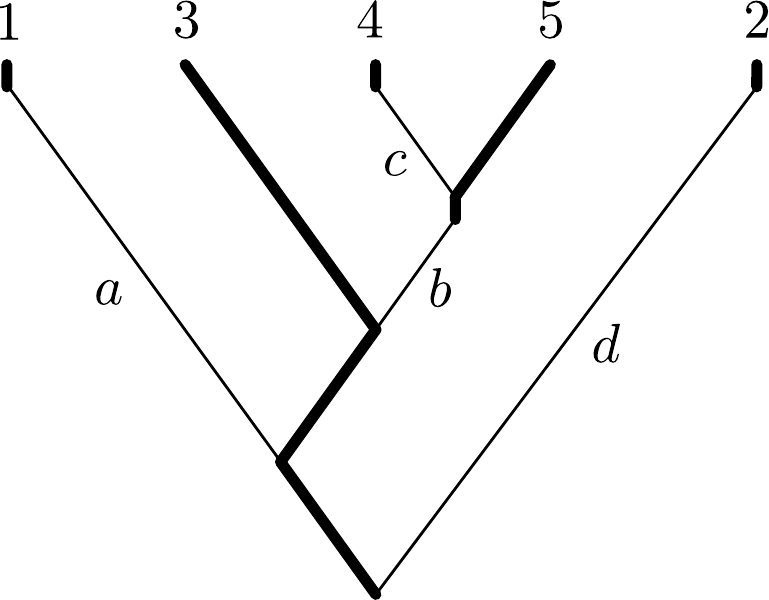}
}} 
\quad
\mapsto
\qquad 
\vcenter{
\xymatrix{&&4\\
1 & 2 & 5 \ar[u]_c \\
 & 3 \ar[ul]^a \ar[u]_d \ar[ur]_b & }
}
\end{equation}
The fact that the quadratic presentation forms a Gr\"obner basis means that the operad it presents may be described by certain admissible tree monomials.  By comparing this basis with the $Y$-trees via the map just described we may see that $\NAP_Y$ is presented by the quadratic basis. A reader interested in combinatorics should compare our construction with one of the well known ``Catalan bijections'' which takes a planar rooted binary tree with $n$ leaves and contracts all left-going edges, thus obtaining a planar rooted tree with $n$ vertices.
\end{remark}

\subsection{The linear operads of based cacti}

\begin{proposition}
Let $(C,\Delta,\epsilon,\gamma)$ be a graded augmented cocommutative coalgebra. The operad $\BCact_C$ is generated by binary operations $C\otimes\k\mathbb{S}_2$; these operations satisfy the relations 
\begin{gather}
c'\circ_1 c''.(23) = (-1)^{|c'||c''|} c''\circ_1 c'\quad  (\text{for homogeneous } c',c''\in C),\label{nap0}\\
c\circ_2\one = \sum c_{(1)}\circ_1 c_{(2)} \quad (\text{for } c\in C)\label{point}
\end{gather}
which suffice to present the operad.
\end{proposition}

\begin{proof}
According to Proposition~\ref{prop:cacti-distr}, the operad $\BCact_C$ is isomorphic to the quotient of $\NAP_C$ by the operadic ideal generated by relations~\eqref{point}. Also, from the proof of Theorem~\ref{thm:NAP-Koszul}, we know that the relations~\eqref{nap0} are the defining relations of~$\NAP_C$, which completes the proof.
\end{proof}

\begin{remark}
For the sake of completeness, let us describe the relations of the Koszul dual operad~$\BCact_C^!$. Its space of generators is $C^*\otimes\k\mathbb{S}_2$; note that $C^*$ is a graded commutative algebra which splits as $\k\one\oplus\overline{C^*}$. The relations are
\begin{equation}\label{right-comb}
c\circ_2\overline{c}=0 \quad \text{for homogeneous } c\in C^*, \overline{c}\in\overline{C^*},
\end{equation}
\begin{multline}
c'\circ_1 c''-(c'c'')\circ_2\one=\\=(-1)^{|c'||c''|}(c''\circ_1 c'-(c''c')\circ_2\one).(23) \quad \text{for homogeneous } c',c''\in C^*.\label{preLie-type}
\end{multline}
Note that for $c=c'=\one$ the relation~\eqref{preLie-type} is precisely the pre-Lie relation. This is not at all surprising, since by combining Theorem~\ref{thm:koszul-dual} with Theorem~\ref{thm:cacti-koszul} below we expect that the $\mathbb{S}$-modules
\begin{equation}
\BCact_C^! \text{ and } \NAP_{\overline{C}}^!\circ(\Perm)^!\simeq\NAP_{\overline{C}}^!\circ\PreLie 
\end{equation}
are isomorphic, and that $\PreLie$ is a suboperad of~$\BCact_C^!$.
\end{remark}

\begin{theorem}\label{thm:cacti-koszul}
For a graded augmented cocommutative coalgebra~$C$, the operad $\BCact_C$ is Koszul, and as $\mathbb{S}$-modules, 
\begin{equation}
\BCact_C\simeq\Perm\circ\NAP_{\overline{C}}.
\end{equation}
\end{theorem}

\begin{proof}
Let us show that $\BCact_C$ is obtained from $\Perm$ and $\NAP_{\overline{C}}$ via a filtered distributive law.

Using the splitting of~$C$ along the augmentation, we can refine the formulae~\eqref{nap0} and~\eqref{point} as follows:
\begin{gather}
\one\circ_1\one.(23)=\one\circ_1\one,\label{perm1}\\
c\circ_1\one=\one\circ_1 c.(23) \quad (\text{for } c\in\overline{C}),\label{distr1}\\
c'\circ_1 c''.(23) = (-1)^{|c'||c''|} c''\circ_1 c'\quad  (\text{for homogeneous } c',c''\in\overline{C}),\label{nap}\\
\one\circ_2\one = \one\circ_1\one,\label{perm2}\\
c\circ_2\one = \sum c_{(1)}\circ_1 c_{(2)} \quad (\text{for }c\in\overline{C}).\label{distr2}
\end{gather}
The formulae~\eqref{perm1}, \eqref{distr1}, and~\eqref{nap} represent the formula~\eqref{nap0} after splitting, and the formulae~\eqref{perm2} and~\eqref{distr2} represent the formula~\eqref{point} after splitting. It is clear that the formulae~\eqref{perm1} and~\eqref{perm2} describe the operad~$\Perm$, while the formula~\eqref{nap} describes precisely the operad~$\NAP_{\overline{C}}$. It remains to show that the formulae~\eqref{distr1} and~\eqref{distr2} define a filtered distributive law between these two operads. To be precise, we first need to check that the formula~\eqref{distr2} stands a chance of defining a distributive law, since \emph{a priori} its right hand side is a mixture of all possible tree monomials. However, we first note that the compatibility of the counit with the coproduct ensures that if $c\in\overline{C}$ then  
\begin{equation}
\Delta(c)\in\overline{C}\otimes\k\one+\k\one\otimes\overline{C}+\overline{C}\otimes\overline{C},
\end{equation}
so the tree monomial $\one\circ_1\one$ is missing on the right hand side of~\eqref{distr2}. Also, the tree monomials of the form $c'\circ_1\one$ (with $c'\in\overline{C}$) appearing on the right hand side should be rewritten using the formula~\eqref{distr1}, but this minor detail will not affect any of our computations.

To check that the formulae~\eqref{distr1} and \eqref{distr2} define a filtered distributive law between $\Perm$ and~$\NAP_{\overline{C}}$, one need to perform carefully all ambiguous rewritings bringing the generator~$\one$ towards the root of a tree monomial, checking that they do not give additional new relations. We shall omit the details, indicating briefly that the rewriting of 
\begin{equation}\label{commutativity}
c\circ_2(\one\circ_1\one.(23))= c\circ_2(\one\circ_1\one) 
\end{equation}
does not result in a new relation because the coproduct of~$C$ is cocommutative, while the rewriting of
\begin{equation}\label{associativity}
c\circ_2(\one\circ_1\one)=c\circ_2(\one\circ_2\one) 
\end{equation}
does not result in a new relation because the coproduct of~$C$ is coassociative, and finally the rewriting of
\begin{equation}\label{nap-rewriting1}
c'\circ_1(c''\circ_2\one)=(-1)^{|c'||c''|}(c''\circ_2\one)\circ_1 c', 
\end{equation}
as well as
\begin{equation}\label{nap-rewriting2}
(c\circ_1\one)\circ_3\one=(c\circ_2\one)\circ_1\one
\end{equation}
does not result in new relations because of the NAP-type relations \eqref{nap0}. This, together with the observation that the projection $\BCact_C\twoheadrightarrow\Perm$ always splits because the relations of $\Perm$ remain undeformed ($s=0$), completes the proof of our theorem.
\end{proof}

\begin{remark}
Let $Y$ be the (pointed) two-element set~$\{\mathbf{0},1\}$, so that $C=H_*(Y)$ is the split two-dimensional coalgebra $\k\oplus\k$, as in the example~\ref{ex:perm-nap} below. Theorem~\ref{thm:cacti-koszul} shows that we have an $\mathbb{S}$-module isomorphism
\begin{equation}
\BCact_C\simeq\Perm\circ\NAP\simeq\Perm\circ\PreLie\simeq\NAP^!\circ\PreLie\simeq\BCact_C^!, 
\end{equation}
but the operads $\BCact_C$ and~$\BCact_C^!$ are substantially different. Of course, there is also a trivial operad structure on the $\mathbb{S}$-module $\Perm\circ\PreLie$ for which the insertion of any $\Perm$-operation into any $\PreLie$-operation is equal to zero; this operad is Koszul and self-dual. It is an open question whether there exist nontrivial self-dual Koszul operad structures on $\Perm\circ\PreLie$ via a distributive law or a filtered distributive law between $\Perm$ and $\PreLie$; such operads would be natural candidates to encode ``pre-Poisson algebras'' (much different from the ones in~\cite{Ag00}) and ``pre-associative algebras''. 
\end{remark}

\bibliographystyle{amsplain}

\providecommand{\bysame}{\leavevmode\hbox to3em{\hrulefill}\thinspace}
\providecommand{\MR}{\relax\ifhmode\unskip\space\fi MR }
\providecommand{\MRhref}[2]{%
  \href{http://www.ams.org/mathscinet-getitem?mr=#1}{#2}
}
\providecommand{\href}[2]{#2}

\end{document}